\documentclass[11pt,reqno]{article}
\usepackage[utf8]{inputenc}
\usepackage{amsthm,amsmath,amssymb,latexsym,soul,cite,mathrsfs}
\usepackage{xcolor}
\usepackage{color,enumitem,graphicx}
\usepackage[colorlinks=true,urlcolor=blue,
citecolor=red,linkcolor=blue,linktocpage,pdfpagelabels,
bookmarksnumbered,bookmarksopen]{hyperref}
\usepackage[english]{babel}
\usepackage{enumitem}
\usepackage[left=2.5cm,right=2.5cm,top=2.9cm,bottom=2.9cm]{geometry}
\usepackage[hyperpageref]{backref}

\makeatletter
\newcommand{\leqnomode}{\tagsleft@true}
\newcommand{\reqnomode}{\tagsleft@false}
\makeatother

\numberwithin{equation}{section}
\newtheorem{theorem}{Theorem}[section]
\newtheorem{lemma}[theorem]{Lemma}
\newtheorem{corollary}[theorem]{Corollary}
\newtheorem{proposition}[theorem]{Proposition}

\newtheorem{remark}[theorem]{Remark}
\newtheorem{definition}[theorem]{Definition}

\newenvironment{acknowledgement}{\textbf{Acknowledgement.}\em}{}

\title{On Fractional Musielak-Sobolev spaces and applications to nonlocal problems}
	 
\author{ J.C. de Albuquerque, L.R.S. de Assis, M.L.M. Carvalho and  A. Salort}

\begin{document}
	
	\maketitle
	
	\begin{abstract}
		In this work, we establish some abstract results on the perspective of the fractional Musielak-Sobolev spaces, such as: uniform convexity, Radon-Riesz property with respect to the modular function, $(S_{+})$-property, Brezis-Lieb type Lemma to the modular function and monotonicity results. Moreover, we apply the theory developed to study the existence of solutions to the following class of nonlocal problems
		\begin{equation*}
			\left\{
			\begin{array}{ll}
				(-\Delta)_{\Phi_{x,y}}^s u = f(x,u),& \mbox{in }\Omega,\\
				u=0,& \mbox{on }\mathbb{R}^N\setminus \Omega,
			\end{array}
			\right.
		\end{equation*}	
		where $N\geq 2$, $\Omega\subset \mathbb{R}^N$ is a bounded domain with Lipschitz boundary $\partial \Omega$ and $f:\Omega \times \mathbb{R} \rightarrow \mathbb{R}$ is a Carath\'{e}odory function not necessarily satisfying the Ambrosetti-Rabinowitz condition. Such class of problems enables the presence of many particular operators, for instance, the fractional operator with variable exponent, double-phase and double-phase with variable exponent operators, anisotropic fractional $p$-Laplacian, among others. 
	\end{abstract}
	
	\bigskip
	
	\noindent{\small\emph{2010 Mathematics Subject Classification:} {\small 46E30; 35R11; 47G20.}
		
	\noindent\emph{Keywords and phrases:}} {\small Fractional Musielak-Sobolev spaces; Monotonicity results; Nonlocal problems.}
	
	\tableofcontents
	
	
	\section{Introduction}
	
	
	
	
	
	\hspace{0,5cm} In the last years, nonlocal problems involving fractional operators and the corresponding fractional Sobolev spaces have been the subject of several researches, both for its theoretical abstract structure and for its concrete applications in different context, such as thin obstacle problem, mathematical finance, phase transitions, optimization, anomalous diffusion, materials science, conservation laws, image processing, minimal surfaces, water waves and many others, see \cite{Di Nezza}. Recently, Azroul et al. \cite{Azroul1,Azroul2} have considered the new fractional Musielak-Sobolev space $W^{s,\Phi_{x,y}}(\Omega)$ and problems driven by nonlocal integro-differential operator of elliptic type defined as follows
	\begin{eqnarray}\label{phi-frac}
		(-\Delta)_{\Phi_{x,y}}^s u(x)\displaystyle:= 2\lim_{\varepsilon\rightarrow 0} \int_{\mathbb{R}^N\setminus B_{\varepsilon}(x)} 	\varphi_{x,y}\left(\frac{|u(x)-u(y)|}{|x-y|^s}\right) \frac{u(x)-u(y)}{|x-y|^s}\frac{\mathrm{d}y}{|x-y|^{N+s}}, \;\; x \in \mathbb{R}^N
	\end{eqnarray}
	 where $s \in (0,1)$, $\Phi(x,y,t)=\int_0^{|t|} \tau \varphi(x,y,\tau)\,\mathrm{d}\tau \; \mbox{and} \; \varphi:\Omega\times\Omega\times (0,+\infty)\rightarrow[0, +\infty)$ are Carath\'{e}odory functions that satisfy some suitable assumptions which will be mentioned later. For convenience of the reader, Section \ref{s0} is devoted to introduce some basic concepts on this subject.
	
	The fractional Musielak-Sobolev space extends many others known concepts in the literature. For instance, in the case that $\Phi$ is independent of $(x,y)$, the operator \eqref{phi-frac} is a natural fractional version of the nonhomogeneous $\Phi$-Laplace operator and the corresponding space is the fractional Orlicz-Sobolev space. To the best of our knowledge, Fern\'andez Bonder and Salort \cite{Bonder} were the first to introduce some results on this context.  For more information on these spaces and applications involving the fractional $\Phi$-Laplace operator, we refer the readers to Azroul et al. \cite{Azroul et al. eigenvalue}, Bahrouni and Ounaies \cite{Bahrouni-Emb}, Bahrouni et al. \cite{A. Bahrouni et al. non-variat.}, Bahrouni et al. \cite{S. Bahrouni et.al}, Fern\'andez Bonder et al. \cite{Bonder et al.}, Salort \cite{Salort eigenvalue} and references therein. On the other hand, when $\Phi$ depends on $(x,y)$, the theory considers also fractional Sobolev spaces with variable exponents and related nonlocal $p(x, \cdot)$-Laplace operator, which is a fractional version of nonhomogeneous $p(x)$-Laplace operator. The first results referring to these fractional spaces and operators were obtained by Kaufmann et al. \cite{Rossi et. al} and complemented by Bahrouni and R\u{a}dulescu \cite{Bahrouni and Radulescu}. In this direction, we also refer the readers to \cite{Azroul0,A. Bahrouni,Biswas et al.,Chammem et al.,Zuo et al.}. Thus, such discussion suggests that the fractional Musielak-Sobolev space extends and unifies the standard fractional Sobolev space, the fractional Sobolev space with variable exponents and the fractional Orlicz-Sobolev space. In this sense, it is natural to study if the known results can be extended to this new setting. As far as we know, the only results on these spaces and operator \eqref{phi-frac} were obtained by Azroul et al. \cite{Azroul1, Azroul2} and Azroul et al. \cite{Azroul3}. 
	
	In \cite{Azroul1}, the authors proved that $W^{s,\Phi_{x,y}}(\Omega)$ is a separable and reflexive Banach space. Moreover, if $\Phi_{x,y}$ satisfies the $\Delta_2$-condition (see Subsection \ref{ss21}) and 
	\begin{equation}\label{hip. convex}
		 t\mapsto\Phi_{x,y}(\sqrt{t}) \;\; \mbox{is convex for all}\; (x,y)\in \Omega\times\Omega,
	\end{equation}
	then $W^{s,\Phi_{x,y}}(\Omega)$ is a uniformly convex space. The proofs are inspired by results obtained in \cite{Mihailescu and Radulescu} for generalized Orlicz-Sobolev spaces. In \cite{S. Bahrouni et.al,A. Bahrouni et al. non-variat., Azroul3}, the convexity assumption \eqref{hip. convex} has been used to obtain the $(S_+)$-property for a suitable functional (see Definition \ref{S+definition}). The $(S_+)$- property plays an important role in the study of solutions for differential equations in Orlicz-Sobolev and Musielak-Sobolev spaces, see for instance \cite{Mihailescu and Radulescu,Xianling Fan,Liu and Zhao, Liu and Dai 1} and references therein. Recently, the $(S_+)$-property was obtained for a very wide class of operators associated to the fractional Orlicz-Sobolev and Musielak-Sobolev spaces, see \cite{Azroul3,S. Bahrouni et.al,A. Bahrouni et al. non-variat.}. In these works, in order to prove that the space is uniformly convex and the $(S_+)$-property, the authors assumed $\Delta_2$-condition and the convexity assumption \eqref{hip. convex}. It is important to mention that hypothesis \eqref{hip. convex} is restrictive, since it excludes classical examples of functions with balanced growth, for instance, $\Phi_{x,y}(t)=|t|^p$ with $1<p<2$, which satisfies $\Delta_2$-condition, but does not satisfy \eqref{hip. convex}. In this work, we will extend these results assuming only that $\Phi$ and its conjugate function satisfy the $\Delta_2$-condition. 
	
	Motivated by the above discussion, our main goal in this paper is extend and complement the previous results on the perspective of the new class of fractional Musielak-Sobolev space $W^{s,\Phi_{x,y}}(\Omega)$ and the related nonlocal integro-differential operator \eqref{phi-frac}. In the abstract point of view, our main contributions are the following:
	\begin{itemize}
		\item[(i)] We prove that $W^{s,\Phi_{x,y}}(\Omega)$ is uniformly convex and the Radon-Riesz property with respect to the modular function;
		\item[(ii)] We prove the $(S_+)$-property;
		\item[(iii)] We prove a Brezis-Lieb type Lemma to the modular function and other convergence results;
		\item[(iv)] We introduce several monotonicity properties of the nonlocal integro-differential operator \eqref{phi-frac};
		\item[(v)] We obtain an existence result for a very general class of nonlocal problems without Ambrosetti-Rabinowitz condition.
	\end{itemize}
	It is important to emphasize that items (i)--(iii) are obtained without assuming the convexity assumption \eqref{hip. convex}.
	
	As mentioned in item (v), we apply the theory developed to study the existence of weak solutions to the following class of nonlocal problems
	\begin{equation}\label{P0}
		\left\{
		\begin{array}{ll}
			(-\Delta)_{\Phi_{x,y}}^s u = f(x,u),& \mbox{in }\Omega,\\
			u=0,& \mbox{on }\mathbb{R}^N\setminus \Omega,
		\end{array}
		\right.
	\end{equation}
	where $N\geq 2$, $\Omega\subset \mathbb{R}^N$ is a bounded domain with Lipschitz boundary $\partial \Omega$, $(-\Delta)_{\Phi_{x,y}}^s$ is the nonlocal integro-differential operator of elliptic type defined in \eqref{phi-frac} and $f:\Omega \times \mathbb{R} \rightarrow \mathbb{R}$ is a Carath\'{e}odory function satisfying some suitable growth conditions. It is worth noting that Problem \eqref{P0} enables the presence of many particular operators, such as, the fractional operator with variable exponent, double-phase and double-phase with variable exponent operators, among others. For instance, we consider the following examples:

	\begin{itemize}
		\item \textit{Fractional Double-Phase operator}: $\Phi_{x,y}(t)=\frac{1}{p}|t|^p+\frac{1}{q}a(x,y)|t|^q$ with $1<p<q<N$ and $a\in L^\infty(\Omega\times \Omega)$ a non-negative symmetric function;
	 
	 	\item \textit{Fractional $p(x,\cdot)$-Laplacian}: $\Phi_{x,y}(t)=\frac{1}{p(x,y)}|t|^{p(x,y)}$ with $p\in C(\Omega\times\Omega)$ symmetric and satisfying $1<p^{-}\leq p(x,y)\leq p^{+}<N$, for all $(x,y)\in \Omega\times\Omega$;
	 
		 \item  \textit{Logarithmic perturbation of the $p(x,\cdot)$-Laplacian}: $\Phi_{x,y}(t)=|t|^{p(x,y)}\log(1+|t|)$ with $p\in C(\Omega\times\Omega)$ symmetric and satisfying $1<p^{-}\leq p(x,y)\leq p^{+}<N-1,\; \mbox{for all}\; (x,y)\in \Omega\times\Omega$;
		 
		 \item \textit{Anisotropic fractional $p$-Laplacian}: $\Phi_{x,y}(t)=a(x,y)|t|^p$, with $1<p<+\infty$ and $a\in L^\infty(\Omega\times \Omega)$ symmetric and satisfying $0<a^{-}\leq a(x,y)\leq a^{+}<\infty$.
	\end{itemize}
	

	\vspace{0,5cm}
	
	The  paper is organized as follows. In the forthcoming section, we collect some preliminary results and we establish some notation that will be used throughout this paper. Section \ref{section 3} is devoted to study monotonicity results for operators in fractional Musielak-Sobolev spaces. In the Section \ref{Existence result}, we apply variational methods to obtain weak solution for problem \eqref{P0}. Finally, in Section \ref{examples}, we present some examples of functions $\Phi$ and nonlinearities $f$ for which the existence result may be applied.
	
	
	
	\section{Preliminaries} \label{s0}	
	
	\subsection{Generalized $N$-functions}\label{ss21}
		
	In this Section we collect preliminary concepts of the theory of Musielak-Orlicz spaces and fractional Musielak-Sobolev spaces, which will be used throughout the paper. For a more complete discussion on this subject we refer the readers to \cite{adams, Azroul1, Azroul2, Harjulehto1, Harjulehto2, Fuk_1, Musielak}.  In this work, unless we specifically point it out, we consider $N\geq 2$, $\Omega$ is a bounded domain in $\mathbb{R}^N$ and $\Phi:\Omega \times \Omega \times\mathbb{R}\rightarrow\mathbb{R}$ is a Carath\'{e}odory function defined by
	$$
	\Phi_{x,y}(t):=\Phi(x,y,t)=\int_0^{|t|} \tau\varphi(x,y,\tau)\;\mathrm{d}\tau,
	$$
	where  $\varphi:\Omega \times \Omega \times(0,+\infty)\rightarrow[0,+\infty)$ is a function satisfying the following assumptions:
	\begin{itemize}
		\item[($\varphi_1$)] $\lim_{t\rightarrow 0} t\varphi_{x,y}(t)=0$ and  $\lim_{t\rightarrow +\infty} t\varphi_{x,y}(t)=+\infty$, where $t \mapsto\varphi_{x,y}(t):=\varphi(x,y,t)$ is a continuous function on $(0,+\infty)$, for all $(x,y) \in \Omega\times \Omega$; 
		\item[($\varphi_2$)] $t\mapsto t\varphi_{x,y}(t)$ is increasing on $(0, +\infty)$;
		\item[($\varphi_3$)] there exist $1<\ell\leq m < +\infty$ such that
		$$
		\ell\leq \frac {t^2\varphi_{x,y}(t)}{\Phi_{x,y}(t)}\leq m, \quad \mbox{for all} \hspace{0,2cm} (x,y)\in \Omega \times\Omega \hspace{0,2cm}  \mbox{and} \hspace{0,2cm} t>0.
		$$
	\end{itemize}
	We also consider the function $\widehat{\Phi}:{\Omega}\times\mathbb{R} \rightarrow\mathbb{R}$ given by 
	\begin{equation}\label{phihat}
	\widehat{\Phi}_x(t):=\widehat{\Phi}(x,t)=\int_0^{|t|} \tau \widehat{\varphi}_x(\tau)\;\mathrm{d} \tau,
	\end{equation}
	where $\widehat{\varphi}_x(t):=\widehat{\varphi}(x,t)=\varphi(x,x,t)$ for all $(x,t)\in \Omega \times (0,+\infty)$. The assumption $(\varphi_3)$ implies that 
	$$
	\ell\leq \frac{t^2 \widehat{\varphi}_x(t)}{\widehat{\Phi}_x(t)}\leq m, \quad \mbox{for all} \hspace{0,2cm} x\in \Omega \hspace{0,2cm}  \mbox{and} \hspace{0,2cm} t>0.
	$$
	
	\begin{remark}
		In contrast to the assumption $(\varphi_3)$, if the convexity hypothesis \eqref{hip. convex} holds, then
		$$\Phi_{x,y}(s)\geq \Phi_{x,y}(t) + \frac{\varphi_{x,y}(t)}{2}(s^2 -t^2 ),\quad \mbox{for all} \hspace{0,2cm} (x,y)\in \Omega \times\Omega \hspace{0,2cm}  \mbox{and} \hspace{0,2cm} s,t>0.$$
		By fixing $t>0$ and making $s\rightarrow 0^+$, we get
		$$2\leq\frac{t^2\varphi_{x,y}(t)}{\Phi_{x,y}(t)},\quad \mbox{for all} \hspace{0,2cm} (x,y)\in \Omega \times\Omega \hspace{0,2cm}  \mbox{and} \hspace{0,2cm} t>0,$$
		which excludes cases where $1<\ell<2$.
	\end{remark}
	
	
	As a consequence of the previous assumptions, one may obtain the following properties:
		\begin{itemize}
			\item[(i)] 
			$\Phi_{x,y}(t)$ is even, continuous, increasing and convex in $t$;
			\item[(ii)] $\displaystyle\lim_{t\rightarrow 0}\frac{\Phi_{x,y}(t)}{t}=0$;
			\item[(iii)] $\displaystyle\lim_{t\rightarrow \infty}\frac{\Phi_{x,y}(t)}{t}=\infty$;
			\item[(iv)] $\Phi_{x,y}(t)>0$, for all $t>0$.	
		\end{itemize}	
	See for instance the book of Kufner-John-Fu\v{c}\'{i}k \cite[Lemma 3.2.2]{Kufner-John-Fucik}.
	\begin{definition}
	A function $\Phi\colon \Omega\times \Omega\times \mathbb{R} \to \mathbb{R}$ is said to be a \emph{generalized $N$-function} if it fulfills the properties (i)--(iv) above for a.e. $(x,y)\in\Omega\times\Omega$, and for each $t\in\mathbb{R}$,  $\Phi_{x,y}(t)$ is measurable in $(x,y)$. We denote by $\mathcal{N}(\Omega\times\Omega)$ the set of all generalized $N$-functions defined on $\Omega\times\Omega$.
	\end{definition}

In light of assumption ($\varphi_3$),   $\Phi_{x,y}$ and $\widehat{\Phi}_x$ satisfy the so-called \emph{$\Delta_{2}$-condition:} there exists $K>0$ such that
\begin{align*}
&\Phi_{x,y}(2t) \leq K \Phi_{x,y}(t), \quad \mbox{for all} \hspace{0,2cm} (x,y)\in \Omega\times\Omega \hspace{0,2cm}  \mbox{and} \hspace{0,2cm} t>0,\\
&\widehat \Phi_{x}(2t) \leq K \widehat\Phi_{x}(t), \quad \mbox{for all} \hspace{0,2cm} x\in \Omega \hspace{0,2cm}  \mbox{and} \hspace{0,2cm} t>0.
\end{align*}
See \cite[Proposition 2.3]{Mihailescu and Radulescu}.
	
	
%

%
%

	\begin{definition}
		For $\Phi \in \mathcal{N}(\Omega \times \Omega)$, the function $\widetilde{\Phi}: \Omega \times \Omega \times \mathbb{R} \rightarrow [0, +\infty)$  defined by
		\begin{equation}\label{j1}
			\widetilde{\Phi}_{x,y}(t)=\widetilde{\Phi}(x,y,t):=\sup_{s\geq0} (ts - \Phi_{x,y}(s)), \quad \mbox{for all} \hspace{0,2cm} (x,y)\in \Omega \times \Omega \hspace{0,2cm} \mbox{and}\hspace{0,2cm} t\in \mathbb{R}
		\end{equation}
		is called the conjugate of $\Phi$ in the sense of Young. 
	\end{definition}
	It is not hard to see that $(\varphi_1)-(\varphi_3)$ imply that $\widetilde{\Phi}$ is a generalized $N$-function and satisfies the $\Delta_{2}$-condition. Furthermore, in view of \eqref{j1} one may deduce the Young’s type inequality 
	\begin{equation}\label{Young's type inequality}
		st\leq \Phi_{x,y}(s) +\widetilde{\Phi}_{x,y}(t), \quad \mbox{for all} \hspace{0,2cm} (x,y)\in\Omega\times\Omega \hspace{0,2cm} \mbox{and} \hspace{0,2cm} s,t\geq0.
	\end{equation}
	Arguing as in \cite[Lemma 2.1]{Azroul1}, one can prove that $\Phi$ and $\widetilde{\Phi}$ satisfy the following
	inequality
	\begin{equation}\label{est-conjugate}
		\widetilde{\Phi}_{x,y}(t\varphi_{x,y}(t))\leq \Phi_{x,y}(2t), \quad \mbox{for all} \hspace{0,2cm} (x,y)\in \Omega\times\Omega \hspace{0,2cm} \mbox{and}\hspace{0,2cm} t\geq0.
	\end{equation}

	
	\subsection{Musielak-Orlicz spaces}
	
	In correspondence to $\widehat{\Phi}_x$, we recall the Musielak-Orlicz space
		$$
		L_{\widehat{\Phi}_x}(\Omega)=\left\{u:\Omega \rightarrow \mathbb{R}~\mbox{measurable}: \int_\Omega \widehat{\Phi}_x(\lambda|u|)\, \mathrm{d}x < \infty, ~ \mbox{for some}~ \lambda>0\right\}.
		$$
		Under our assumptions, $L_{\widehat{\Phi}_x}(\Omega)$ is a separable and reflexive Banach space when endowed with the Luxemburg norm
		$$
		\|u\|_{\widehat{\Phi}_x}:=\inf\left\{\lambda>0:\int_\Omega \widehat{\Phi}_x\left(\frac{u}{\lambda}\right)\,\mathrm{d}x\leq 1\right\}.
		$$ 
		Using the Young’s type inequality \eqref{Young's type inequality} for $\widehat{\Phi}$ and $\widetilde{\widehat{\Phi}}$, one may deduce the following H\"{o}lder's type inequality 
		$$\left| \int_{\Omega} u v \; \mathrm{d}x \right| \leq 2\|u\|_{\widehat{\Phi}_x} \|v\|_{\widetilde{\widehat{\Phi}}_x},$$
		for all $u\in L_{\widehat{\Phi}_x}(\Omega)$ and $v \in L_{\widetilde{\widehat{\Phi}}_x}(\Omega)$, see \cite[Lemma 2.6.5]{Harjulehto1}. For more details on Musielak-Orlicz space, see Diening et al. \cite{Harjulehto1} and Musielak \cite{Musielak}.
		
		Now, we introduce the modular functions $J_{\widehat{\Phi}_x}:L_{\widehat{\Phi}_x}(\Omega)\rightarrow \mathbb{R}$ and $J_{\widetilde{\widehat{\Phi}}_x}: L_{\widetilde{\widehat{\Phi}}_x}(\Omega)\rightarrow \mathbb{R}$ defined by
		$$J_{\widehat{\Phi}_x}(u)=\int_{\Omega} \widehat{\Phi}_x(|u(x)|)\,\mathrm{d}x\quad \mbox{and}\quad J_{\widetilde{\widehat{\Phi}}_x}(u)=\int_{\Omega}\widetilde{ \widehat{\Phi}}_x(|u(x)|)\,\mathrm{d}x .$$
	Henceforth, we use the following notation:
	 $$
	 \begin{aligned}
	 	&\xi_0^-(t)=\min\{t^\ell,t^m\}, \quad \xi_0^+(t)=\max\{t^\ell,t^m\},\\
	 	&\xi_1^{-}(t)=\min\{t^{\widetilde{\ell}},t^{\widetilde{m}}\}, \quad \xi_1^{+}(t)=\max\{ t^{\widetilde{\ell}}, t^{\widetilde{m}}\} ,~~ t\geq 0,
	 \end{aligned}
	 $$
	 where $\widetilde{\ell}=\frac{\ell}{\ell -1}$ and $\widetilde{m}=\frac{m}{m -1}$. Finally, we recall the following Lemma due to Fukagai et al. \cite[Lemmas 2.1 and 2.5]{Fuk_1}:
	\begin{lemma}\label{M-N1}
		Assume that $(\varphi_1)-(\varphi_3)$ hold. Then, $\widehat{\Phi}$ and $\widetilde{\widehat{\Phi}}$ satisfy the following estimates:
		\begin{itemize}
			\item [(i)] $\xi_0^-(\sigma)\widehat{\Phi}_x(t)\leq\widehat{\Phi}_x(\sigma t)\leq \xi_0^+(\sigma)\widehat{\Phi}_x(t)$, for all $x \in \Omega$ and $\sigma, t\geq 0$; 
			
			\item[(ii)] $\xi_0^-(\|u\|_{\widehat{\Phi}_x})\leq J_{\widehat{\Phi}_x}(u)\leq \xi_0^+(\|u\|_{\widehat{\Phi}_x})$, for all $u\in L_{\widehat{\Phi}_x}(\Omega);$
			
			\item[(iii)] $\xi_1^-(\sigma)\widetilde{\widehat{\Phi}}_x(t)\leq\widetilde{\widehat{\Phi}}_x(\sigma t)\leq \xi_1^+(\sigma)\widetilde{\widehat{\Phi}}_x(t)$, for all $x\in \Omega$ and $\sigma, t\geq 0$; 
			
			\item[(iv)] $\xi_1^-(\|u\|_{\widetilde{\widehat{\Phi}}_x})\leq J_{\widetilde{\widehat{\Phi}}_x}(u)\leq \xi_1^+(\|u\|_{\widetilde{\widehat{\Phi}}_x})$, for all $u\in L_{\widetilde{\widehat{\Phi}}_x}(\Omega).$
		\end{itemize}
	\end{lemma}


	\subsection{Fractional Musielak-Sobolev spaces}
	Let $\Phi \in \mathcal{N}(\Omega\times \Omega)$ and $s\in (0,1)$. The fractional Musielak-Sobolev space is defined as follows
	$$
	W^{s,\Phi_{x,y}}(\Omega)=\left\{u\in L_{\widehat{\Phi}_x}(\Omega): J_{s,\Phi}(\lambda u)<\infty, ~ \mbox{for some}~ \lambda>0\right\},
	$$ 
	where 
	$$
	J_{s,\Phi}(u):=\int_{\Omega}\int_{\Omega} \Phi_{x,y}\left( D_s u(x,y)\right)d\mu,\quad \mbox{for}~ s\in (0,1),
	$$
	with
	$$
	\mathrm{d}\mu:= \frac{\mathrm{d}x\mathrm{d}y}{|x-y|^{N}}\quad \mbox{and} \quad D_su(x,y):=\frac{u(x)-u(y)}{|x-y|^s}.
	$$ 
	It is well known that $\mathrm{d}\mu$ is a regular Borel measure on the set $\Omega\times \Omega$.

	The space $W^{s,\Phi_{x,y}}(\Omega)$ is endowed with the norm
	\[
	\|u\|_{W^{s,\Phi_{x,y}}(\Omega)}:=\|u\|_{\widehat{\Phi}_x}+[u]_{s,\Phi_{x,y}},
	\]  
	where the term $[ \cdot ]_{s,\Phi_{x,y}}$ is the so called $(s,\Phi_{x,y})$-\textit{Gagliardo seminorm} defined by
	\[
	[u]_{s,\Phi_{x,y}}:=\inf\left\{\lambda>0:J_{s,\Phi}\left(\frac{u}{\lambda}\right)\leq 1\right\}.
	\]
	It is important to emphasize that assumption $(\varphi_3)$ implies that $\Phi$ and $\widetilde{\Phi}$ satisfy the $\Delta_2$-condition. For this reason, $W^{s,\Phi_{x,y}}(\Omega)$ is a reflexive and separable Banach space, see \cite[Theorem 2.1]{Azroul1}. Note that for the case $\Phi_{x,y}(t)=\Phi(t)$, i.e., when $\Phi$ is independent of the variables $x$ and $y$, we have that $L_{\Phi}(\Omega)$ and $W^{s,\Phi}(\Omega)$ are Orlicz spaces and fractional Orlicz-Sobolev spaces respectively, see \cite{Bonder}. 
	
	We denote by $\left\langle\cdot, \cdot \right\rangle$ the duality pairing between $W^{s,\Phi_{x,y}}(\Omega)$ and its dual
	space $\big(W^{s,\Phi_{x,y}}(\Omega)\big)^\ast$.


	\begin{definition}
		We say that $\Phi \in \mathcal{N}(\Omega\times \Omega)$ satisfies the fractional boundedness condition if
		\begin{equation}\label{j2}\tag{$\mathcal{B}_f$}
			0 <C_1\leq\Phi_{x,y}(1)\leq C_2, \quad \mbox{for all}\hspace{0.2cm} (x,y) \in \Omega \times \Omega,
		\end{equation}
	for some constants $C_1$ and $C_2$.
	\end{definition}

Due to condition \eqref{j2}, for any $s\in (0,1)$ it holds that $C_0^2(\Omega)\subset W^{s,\Phi_{x,y}}(\Omega)$, see \cite[Theorem 2.2]{Azroul1}.

	
	A fractional version of Lemma \ref{M-N1} can be stated as follows:
	
	\begin{lemma}(\cite[Lemma 2.2 and Proposition 2.3 ]{Azroul1})\label{M-N2}
		Assume that $(\varphi_1)-(\varphi_3)$ hold and let $s\in(0,1)$. Then, the following assertions hold:
		\begin{itemize}
			\item [(i)] $\xi_0^-(\sigma)\Phi_{x,y}(t)\leq \Phi_{x,y}(\sigma t)\leq \xi_0^+(\sigma)\Phi_{x,y}(t)$, for all $(x,y) \in \Omega\times \Omega$ and $\sigma, t\geq 0$;
			\item[(ii)] $\xi_0^-([u]_{s,\Phi})\leq J_{s,\Phi}(u)\leq \xi_0^+([u]_{s,\Phi})$, for all $u\in W^{s,\Phi_{x,y}}(\Omega)$.
		\end{itemize}
	\end{lemma}

In particular, this lemma together with \eqref{j2} gives that
\begin{align*}
&\Phi_{x,y}(t) \leq \xi^+_0(t) \Phi_{x,y}(1) \leq \xi^+_0(t) \sup_{x,y} \Phi_{x,y}(1)<\infty,\\
&\Phi_{x,y}(t) \geq \xi^-_0(t) \Phi_{x,y}(1) \geq  \xi^-_0(t) \inf_{x,y} \Phi_{x,y}(1)>0,~t>0.
\end{align*}

	Let $s \in (0, 1)$ and $\widehat{\Phi}_{x}$ be defined in \eqref{phihat}. For each $x \in \Omega$, the function $\widehat{\Phi}_x:[0,+\infty) \rightarrow [0, +\infty)$ is an increasing homeomorphism. Throughout  this work, we denote by $\widehat{\Phi}_x^{-1}$ the inverse function of $\widehat{\Phi}_x$ and we assume the following conditions:
	\begin{eqnarray}\label{cond-crit-1}
		\int_0^1\frac{\widehat{\Phi}_x^{-1}(\tau)}{\tau^{\frac{N+s}{N}}}\,\mathrm{d}\tau<\infty \quad \mbox{and} \quad\int_1^\infty\frac{\widehat{\Phi}_x^{-1}(\tau)}{\tau^{\frac{N+s}{N}}}\,\mathrm{d}\tau=\infty, \quad \mbox{for all}\; x\in \Omega.
	\end{eqnarray}
	We define the inverse Musielak-Sobolev conjugate function of $\widehat{\Phi}$, denoted by $\widehat{\Phi}_{s,x}^\ast$, as follows:
	$$
	 (\widehat{\Phi}_{s,x}^\ast)^{-1}(t):=(\widehat{\Phi}_s^\ast)^{-1}(x,t)= \int_0^t\frac{\widehat{\Phi}_{x}^{-1}(\tau)}{\tau^{\frac{N+s}{N}}}\,\mathrm{d}\tau, \quad \mbox{for } x\in\Omega \hspace{0,2cm} \mbox{and} \hspace{0,2cm} t\geq0.
	$$
	Let us introduce the notation
	 \[
	 \xi_2^{-}(t)=\min\{t^{\ell_s^\ast}, t^{m_s^\ast}\} \quad \mbox{and} \quad \xi_2^{+}(t)=\max\{t^{\ell_s^\ast}, t^{m_s^\ast}\}, \quad \mbox{for} \hspace{0,2cm} t\geq0,
	 \]
	where $\ell_s^\ast=\frac{N\ell}{N-s\ell}$ and $m_s^\ast=\frac{Nm}{N-sm}$  whenever $\ell, m \in (1,N/s)$. Proceeding as in \cite[Lemma 2.2]{Fuk_1}, we obtain the following result:
	\begin{lemma}\label{M-N critical}
		Assume that $(\varphi_1)-(\varphi_3)$ hold with $\ell, m \in (1,N/s)$ and let $s\in(0,1)$. The following assertions hold:
		\begin{itemize}
			\item[(i)] $\xi_2^{-}(\sigma)\widehat{\Phi}_{s,x}^\ast(t)\leq \widehat{\Phi}_{s,x}^\ast(\sigma t)\leq\xi_2^{+}(\sigma)\widehat{\Phi}_{s,x}^\ast(t)$, for all $x\in \Omega$ and $\sigma,t\geq0$;
			
			\item[(ii)] $\xi_2^{-}(\|u\|_{\widehat{\Phi}_{s,x}^\ast})\leq \displaystyle\int_{\Omega}\widehat{\Phi}_{s,x}^\ast (|u(x)|)\;\mathrm{d}x\leq\xi_2^{+}(\|u\|_{\widehat{\Phi}_{s,x}^\ast})$, for all $u\in L_{\widehat{\Phi}_{s,x}^\ast}(\Omega)$.
		\end{itemize}
	\end{lemma}
	
	Next, we present some embedding results of the  fractional Musielak-Sobolev spaces. 
	
	\begin{lemma}(\cite[Lemma 2.3]{Azroul2})\label{embedding-classical}
		Let $0<s'<s<1$ and $\Omega$ be a bounded domain in $\mathbb{R}^N$. Assume that $(\varphi_1)-(\varphi_3)$ and \eqref{j2} hold. Then, the embedding $W^{s,\Phi_{x,y}}(\Omega)\hookrightarrow W^{s',q}(\Omega)$ is continuous for all $q\in[1, \ell)$.
	\end{lemma}
	
	\begin{remark}\label{embed-comp-part}
		In light of Lemma \ref{embedding-classical} and the classical theory of fractional Sobolev spaces, if $\Omega \subset\mathbb{R}^N$ is a bounded domain with $C^{0,1}$-regularity, then the embedding $W^{s,\Phi_{x,y}}(\Omega)\hookrightarrow L^1(\Omega)$ is compact, see \cite[Theorem 7.1]{Di Nezza}. 
	\end{remark}
	%
	\begin{definition}
		Let $\Phi, \Psi \in \mathcal{N}(\Omega)$. We say that $\Psi$ essentially grows more slowly than $\Phi$ near infinity, and we write $\Psi\ll \Phi$, if for all $k>0$, there holds
		$$\lim_{t\rightarrow +\infty} \frac{\Psi(x,kt)}{\Phi(x,t)}=0, \quad \mbox{uniformly for} \; x\in\Omega.$$
	\end{definition}
	
	\begin{proposition}(\cite[Theorems 2.1 and 2.2]{Azroul2})\label{imbed-critical-bounded}
		Let $s\in(0,1)$ and $\Phi \in \mathcal{N}(\Omega\times \Omega)$ satisfying $(\varphi_3)$, where $\Omega$ is a bounded domain in $\mathbb{R}^N$ with $C^{0,1}$-regularity and bounded boundary. 
		\begin{itemize}
			\item [(i)] If \eqref{j2} and \eqref{cond-crit-1} hold, then the embedding
			$W^{s,\Phi_{x,y}}(\Omega)\hookrightarrow L_{\widehat{\Phi}_{s,x}^\ast}(\Omega)$
			is continuous;
			\item [(ii)] Moreover, for any $\Psi \in \mathcal{N}(\Omega)$ such that $\Psi \ll \widehat{\Phi}_s^\ast$, the embedding
			$W^{s,\Phi_{x,y}}(\Omega)\hookrightarrow L_{\Psi_x}(\Omega)$
			is compact.
		\end{itemize}		
	\end{proposition}

	We consider the following closed subspace of $W^{s,\Phi_{x,y}}(\Omega)$ defined by
	$$W_0^{s,\Phi_{x,y}}(\Omega)=\{u\in W^{s,\Phi_{x,y}}(\mathbb{R}^N): u=0~\mbox{a.e. in }\mathbb{R}^N\setminus \Omega\}$$
	where $\Omega \subset\mathbb{R}^N$ is a domain (bounded or not). We have the following generalized Poincar\'{e} type inequality.
	
	\begin{proposition}(\cite[Theorem 2.3]{Azroul2})\label{P-I}
		Let $s\in(0,1)$ and $\Omega$ be a bounded domain in $\mathbb{R}^N$ with $C^{0,1}$-regularity and bounded boundary. Assume that $(\varphi_1)-(\varphi_3)$ and \eqref{j2} hold. Then, there exists a positive constant $C$ such that
		$$\|u\|_{\widehat{\Phi}_x}\leq C [u]_{s,\Phi_{x,y}},$$
		for all $u \in 	W_0^{s,\Phi_{x,y}}(\Omega)$.
	\end{proposition}
	
	In view of Proposition \ref{P-I}, there exists a positive constant $\lambda_1$ such that 
	\begin{equation}\label{poincareconstant}
		\int_{\Omega} \widehat{\Phi}_x(|u(x)|)\;\mathrm{d}x\leq \lambda_1 \int_{\Omega}\int_{\Omega} \Phi_{x,y}\left(D_su(x,y)\right)\;\mathrm{d}\mu,
	\end{equation}
	for all $u\in	W_0^{s,\Phi_{x,y}}(\Omega)$. Moreover, $[\cdot]_{s,\Phi_{x,y}}$  is a norm on $W_0^{s,\Phi_{x,y}}(\Omega)$ which is equivalent to the usual norm $\|\cdot\|_{s,\Phi_{x,y}}$. For more details on this subject we refer the readers to \cite{Azroul2}. For the sake of simplicity, we use the notations $\|u\|:=[u]_{s,\Phi_{x,y}}.$

	\begin{remark}
		Throughout this work, we need to use some standard tools, such as: Poincar\'{e} type inequality and compactness results for embeddings in fractional Musielak--Orlicz--Sobolev space. For this reason, we shall assume that condition \eqref{j2} is satisfied. Consequently, by $\Delta_2$-condition, $\Phi_{x,y}(k)$, $\widehat{\Phi}_x(k)$, $\widetilde{\widehat{\Phi}}_x(k)$ and $\widehat{\Phi}_{s,x}^\ast(k)$ are bounded for each $k>0$.
	\end{remark}
	
	\section{Monotonicity and convergence results for operators in fractional Musielak-Sobolev spaces}\label{section 3}
	
	
	We start this Section by recalling some definitions introduced in \cite{Harjuleto-Hasto, Harjulehto2} which are needed to prove the uniform convexity of space $W^{s,\Phi_{x,y}}(\Omega)$ and Radon-Riesz property with respect to the modular function.
	\begin{definition} 
		A function $g:(0, +\infty) \rightarrow \mathbb{R}$ is said to be almost increasing if there exists a constant $a\geq1$ such that $g(s)\leq ag(t)$, for all $0<s<t.$ In a similar way, we define almost decreasing.
	\end{definition}
	
	\begin{definition}
		We say that a function $\Phi:\Omega\times\Omega\times [0,+\infty) \rightarrow [0,+\infty]$ is a generalized $\varPhi$-prefunction if $(x,y)\mapsto\Phi_{x,y}(|f(x,y)|)$ is measurable for all measurable function $f:\Omega\times\Omega \rightarrow \mathbb{R}$,
		$$\lim_{t\rightarrow 0^+} \Phi_{x,y}(t)=\Phi_{x,y}(0)= 0\quad \mbox{and}\quad \lim_{t\rightarrow+\infty}\Phi_{x,y}(t)=+\infty, \quad \mbox{for all} \hspace{0,2cm} (x,y)\in\Omega\times\Omega.$$
		A generalized $\varPhi$-prefunction $\Phi$ is said to be a
		\begin{itemize}
			\item[(i)] weak $\varPhi$-function if
			$$t\mapsto\frac{\Phi_{x,y}(t)}{t} \;\mbox{is almost increasing}\;\mbox{on}\; (0,+\infty),\quad \mbox{for all} \hspace{0,2cm} (x,y)\in\Omega\times\Omega;$$
			\item[(ii)] convex  $\varPhi$-function if $t\mapsto\Phi_{x,y}(t)$ is left-continuous and convex for all $(x,y)\in \Omega\times\Omega$.
		\end{itemize}
		The sets of weak and convex $\varPhi$-function are denoted by $\varPhi_w(\Omega\times\Omega)$ and $\varPhi_c(\Omega\times\Omega)$ respectively.
	\end{definition}
	
	\begin{remark}
		We point out that any generalized $N$-function is a weak $\varPhi$-function.
	\end{remark}
	
	\begin{definition}
		A function $\Phi\in \varPhi_c(\Omega\times\Omega)$ is said to be uniformly convex if for given $\varepsilon>0$, there
		exists $\delta:=\delta(\varepsilon) \in(0,1)$ such that
		$$\Phi_{x,y}\left(\frac{s+t}{2}\right)\leq (1-\delta)\frac{\Phi_{x,y}(s)+\Phi_{x,y}(t)}{2}, \quad \mbox{for all} \hspace{0,2cm} (x,y)\in\Omega\times\Omega,$$
		whenever $s,t\geq0$ and $|s-t|\geq\varepsilon\max\{s, t\}$. 
	\end{definition}
	
	Finally, we are able to prove our first main result, which can be stated as follows.
	\begin{theorem}\label{unifor-conv}
		Let $s\in(0,1)$ and assume that $(\varphi_1)-(\varphi_3)$ hold. Then, the following assertions hold: 
		\begin{itemize}
			\item[(i)]$W^{s,\Phi_{x,y}}(\Omega)$ is a uniformly convex space;
			\item[(ii)] If $u_n\rightharpoonup u$ in $W_0^{s,\Phi_{x,y}}(\Omega)$ and $J_{s,\Phi}(u_n)\rightarrow J_{s,\Phi}(u)$, then $u_n\rightarrow u$ in $W_0^{s,\Phi_{x,y}}(\Omega)$.
		\end{itemize}
	\end{theorem}
	\begin{proof}
		\noindent$(i)$ To prove the uniform convexity of the space $W^{s,\Phi_{x,y}}(\Omega)$, let us consider the linear operator $T:W^{s,\Phi_{x,y}}(\Omega) \rightarrow L_{\widehat{\Phi}_x}(\Omega)\times L_{\Phi_{x,y}}(\Omega\times\Omega,\mathrm{d}\mu)$ defined by
		$T(u)=\left(u, D_su\right).$
		It is not hard to see that $T$ is well-defined and it is an isometry. Thus, $T\left(W^{s,\Phi_{x,y}}(\Omega)\right)$ is a closed subspace of $L_{\widehat{\Phi}_x}(\Omega)\times L_{\Phi_{x,y}}(\Omega\times\Omega,\mathrm{d}\mu)$. Since $L_{\widehat{\Phi}_x}(\Omega)$ and $L_{\Phi_{x,y}}(\Omega\times\Omega,\mathrm{d}\mu)$ are Banach spaces, in order to prove that $W^{s,\Phi_{x,y}}(\Omega)$ is uniformly convex space, it is sufficient to show that the spaces $L_{\widehat{\Phi}_x}(\Omega)$ and $L_{\Phi_{x,y}}(\Omega\times\Omega,\mathrm{d}\mu)$ are  uniformly convex. Firstly, by using the condition $(\varphi_3)$ we obtain
		$$\begin{aligned}
			\frac{\mathrm{d}}{\mathrm{d}t}\left(\frac{\Phi_{x,y}(t)}{t^\ell}\right)&=\frac{t^\ell\Phi_{x,y}'(t) - \ell t^{\ell-1}\Phi_{x,y}(t)}{t^{2\ell}}\\
			& = \frac{t^{\ell+1}\varphi_{x,y}(t) - \ell t^{\ell-1}\Phi_{x,y}(t)}{t^{2\ell}}	\\	
			&\geq  \frac{t^{\ell +1}\varphi_{x,y}(t) - t^{\ell+1}\varphi_{x,y}(t)}{t^{2\ell}}=0,  \quad t>0,
		\end{aligned}$$
		which implies that the function $t\mapsto \frac{\Phi_{x,y}(t)}{t^\ell}$ is increasing on $(0,+\infty)$, for all $(x,y)\in\Omega\times\Omega$. Thus, it follows from \cite[Theorem 4.1]{Harjuleto-Hasto} that there exists a uniformly convex function $\Psi \in \varPhi_c(\Omega\times \Omega)$ such that
		$$\Phi_{x,y}(\lambda^{-1}t)\leq\Psi_{x,y}(t)\leq\Phi_{x,y}(\lambda t), \quad \mbox{for all} \hspace{0,2cm} (x,y)\in\Omega\times\Omega \hspace{0,2cm} \mbox{and} \hspace{0,2cm} t\geq0,$$
		for some $\lambda>1$. This fact combined with Lemma \ref{M-N2} $(i)$ imply in 
		$$\lambda^{-m}\Phi_{x,y}(t)\leq\Psi_{x,y}(t)\leq\lambda^m\Phi_{x,y}(t), \quad \mbox{for all} \hspace{0,2cm} (x,y)\in\Omega\times\Omega \hspace{0,2cm} \mbox{and} \hspace{0,2cm} t\geq0.$$
		Hence, the generalized $N$-functions $\Phi_{x,y}$ and $\widehat{\Phi}_x$ are also uniformly convex. Therefore, by \cite[Theorems 2.4.11 and 2.4.14]{Harjulehto1} we conclude that $L_{\widehat{\Phi}_{x}}(\Omega)$ and $L_{\Phi_{x,y}}(\Omega\times\Omega, \mathrm{d}\mu)$ are uniformly convex.
		
		$(ii)$ Since $W_0^{\Phi_{x,y}}(\Omega)$  is a closed subspace of $W^{\Phi_{x,y}}(\Omega)$, we have that $W_0^{\Phi_{x,y}}(\Omega)$ is uniformly convex. By recalling that $\Phi$ satisfies $\Delta_2$-condition, the required property follows from \cite[Lemma 2.4.17 and Remark 2.4.19]{Harjulehto1}.
	\end{proof}
	
	\begin{remark}
	In Theorem \ref{unifor-conv} $(i)$ the domain $\Omega$ could be unbounded or $\mathbb{R}^N$ itself.
	\end{remark}
	
	Inspired by \cite{A. Bahrouni et al. non-variat.}, we introduce a version of the classical Brezis-Lieb's Lemma \cite{Brezis-Lieb}, to modular functions. 
	\begin{proposition}[Brezis-Lieb type Lemma]
		Assume that  $(\varphi_1)-(\varphi_3)$ hold. Let $\{u_n\}_{n\in\mathbb{N}} $ be a bounded sequence in $ W_0^{s,\Phi_{x,y}}(\Omega)$ such that $u_n(x) \rightarrow u(x)$ a.e. in $\mathbb{R}^N$. Then,  $u \in W_0^{s,\Phi_{x,y}}(\Omega)$ and
		$$\lim_{n\rightarrow +\infty} \left( J_{s,\Phi}(u_n) - J_{s,\Phi}(u_n - u) \right) = J_{s,\Phi}(u).$$
	\end{proposition}
	\begin{proof}
		Firstly, by the boundedness of $\{u_n\}_{n\in\mathbb{N}}$, Fatou's Lemma and Lemmas \ref{M-N1} $(ii)$ and \ref{M-N2} $(ii)$, we have
		$$\int_\Omega \widehat{\Phi}_x(|u|) \; \mathrm{d}x \leq \liminf_{n\rightarrow +\infty} \int_\Omega \widehat{\Phi}_x(|u_n|) \; \mathrm{d}x <+\infty,$$
		i.e., $u\in L_{\widehat{\Phi}_x}(\Omega)$, and
		$$\int_\Omega\int_\Omega {\Phi}_{x,y}(|D_su|) \; \mathrm{d}\mu \leq \liminf_{n\rightarrow +\infty}\int_\Omega\int_\Omega {\Phi}_{x,y}(|D_su_n|) \; \mathrm{d}\mu  <+\infty.$$
		By using the property that $u_n$ = 0 a.e. in $\mathbb{R}^N\setminus\Omega$ for all $n\in\mathbb{N}$, it is easy to see that $u$ = 0 a.e. in $\mathbb{R}^N\setminus\Omega$. Hence, $u \in W_0^{s,\Phi_{x,y}}(\Omega)$.

		Now, in view of the Mean Value Theorem, for each $(x,y)\in \Omega\times\Omega$, there exists $z_n:=z_n(x,y)$ between $|D_su_n(x,y) - D_su(x,y)|$ and $|D_su_n(x,y)|$ such that 
		$$\left|\Phi_{x,y}(|D_su_n|) - \Phi_{x,y}(|D_su_n - D_su|)\right| = z_n\varphi_{x,y}(z_n)||D_su_n|-|D_su_n - D_su||,$$
		where we have used that $\Phi_{x,y}'(t)=t\varphi_{x,y}(t)$, for all $t\geq0$. Thus, by using  $(\varphi_2)$, we have
		$$\begin{aligned}
			\left|\Phi_{x,y}(|D_su_n|) - \Phi_{x,y}(|D_su_n - D_su|)\right| & \leq z_n\varphi_{x,y}(z_n)|D_su|\\
			& \leq (|D_su_n|+|D_su_n - D_su|) \varphi_{x,y}(|D_su_n|+|D_su_n - D_su|)|D_su|\\
		\end{aligned}$$
		For any $\varepsilon\in(0,1)$, the Young's inequality \eqref{Young's type inequality} and \eqref{est-conjugate} imply in
		$$\begin{aligned}
			(|D_su|&+|D_su_n - D_su|) \varphi_{x,y}(|D_su|+|D_su_n - D_su|)|D_su|\\
			&\leq \varepsilon\widetilde{\Phi}_{x,y}\big((|D_su|+|D_su_n - D_su|) \varphi_{x,y}(|D_su|+|D_su_n - D_su|)\big) + C_{\varepsilon}\Phi_{x,y}(|D_su|)\\
			& \leq \varepsilon 2^{m} \Phi_{x,y}(|D_su|+|D_su_n - D_su|) +C_{\varepsilon}\Phi_{x,y}(|D_su|)\\
			& \leq  \varepsilon C_{m}\Phi_{x,y}(|D_su_n - D_su|) + C_{\varepsilon,m}\Phi_{x,y}(|D_su|),
		\end{aligned}$$
		where $C_{m}:= 2^{2m-1}$ and $C_{\varepsilon,m}:=\varepsilon 2^{2m-1}+C_{\varepsilon}$. Therefore, we obtain
		\begin{equation}\label{est-BL1}
			\left|\Phi_{x,y}(|D_su_n|) - \Phi_{x,y}(|D_su_n - D_su|)\right| \leq \varepsilon C_{m} \Phi_{x,y}(|D_su_n - D_su|) + C_{\varepsilon,m}\Phi_{x,y}(|D_su|).
		\end{equation}
		Next, for $n\in\mathbb{N}$, we define
		$$\mathcal{W}_{\varepsilon,n}(x,y):= \big[ \left|\Phi_{x,y}(|D_su_n|) - \Phi_{x,y}(|D_su_n - D_su|) - \Phi_{x,y}(|D_su|) \right| - \varepsilon C_{m}\Phi_{x,y}(|D_su_n - D_su|) \big]^{+}, $$ 
		where $a^{+}:=\max\{a,0\},$ for all $a\in\mathbb{R}$. Note that $\mathcal{W}_{\varepsilon,n}(x,y) \rightarrow 0$, as $n\rightarrow +\infty$, a.e. in $\Omega\times\Omega$. Moreover, it follows from \eqref{est-BL1} that
		$$\begin{aligned}
			\left|\Phi_{x,y}(|D_su_n|) - \Phi_{x,y}(|D_su_n - D_su|) - \Phi_{x,y}(|D_su|) \right| &\leq \left|\Phi_{x,y}(|D_su_n|) - \Phi_{x,y}(|D_su_n - D_su|)\right|+ |\Phi_{x,y}(|D_su|)|\\
			& \leq \varepsilon C_{m} \Phi_{x,y}(|D_su_n - D_su|) + (C_{\varepsilon,m} + 1)\Phi_{x,y}(|D_su|),
		\end{aligned}$$
		which implies that 
		$$\mathcal{W}_{\varepsilon,n}(x,y)|x-y|^{-N} \leq (C_{\varepsilon,m} + 1)\Phi_{x,y}(|D_su|)|x-y|^{-N}  \in L^{1}(\Omega\times\Omega).$$
		Hence, in light of  Lebesgue's Dominated Convergence Theorem, there holds
		$$\int_{\Omega}\int_{\Omega} \mathcal{W}_{\varepsilon,n}(x,y)\; \mathrm{d}\mu \rightarrow 0, \quad \mbox{as } n\rightarrow\infty.$$
		This fact and the following inequality 
		$$\begin{aligned}
				| J_{s,\Phi}(u_n) - J_{s,\Phi}(u_n - u) - J_{s,\Phi}(u)|&\leq\int_{\Omega}\int_{\Omega} \left| \Phi_{x,y}(|D_su_n|) - \Phi_{x,y}(|D_su_n - D_su|) - \Phi_{x,y}(|D_su|) \right|\; \mathrm{d}\mu\\
				&\leq \int_{\Omega}\int_{\Omega} \left(  \mathcal{W}_{\varepsilon,n}(x,y)+\varepsilon C_{m} \Phi_{x,y}(|D_su_n - D_su|) \right)\; \mathrm{d}\mu\\
				& \leq \int_{\Omega}\int_{\Omega}  \mathcal{W}_{\varepsilon,n}(x,y) \;\mathrm{d}\mu + \varepsilon C_{m}J_{s,\Phi}(u_n - u), 
			\end{aligned}$$
		imply that
		$$\lim_{n\rightarrow +\infty} | J_{s,\Phi}(u_n) - J_{s,\Phi}(u_n - u) - J_{s,\Phi}(u)| \leq \varepsilon C_{m} K,$$
		for some constant $K>0$. Therefore, by making $\varepsilon \rightarrow 0$ we obtain the desired result.
	\end{proof}
	
	Due to Lemma \ref{M-N2} $(ii)$ and Brezis-Lieb-type Lemma we obtain the following convergence result:
	\begin{corollary}
		 Assume that $(\varphi_1)-(\varphi_3)$ hold. Let $u,u_n\in W_0^{s,\Phi_{x,y}}(\Omega)$, $n \in\mathbb{N}$.   Then, the following assertions are equivalent:
		\begin{itemize}
			\item[(i)] $\displaystyle\lim_{n\rightarrow +\infty} \|u_n - u\| = 0;$
			\item[(ii)] $\displaystyle\lim_{n\rightarrow +\infty} J_{s,\Phi}(u_n - u) = 0;$
			\item[(iii)] $u_n(x) \rightarrow u(x)$ for a.e  $x\in \Omega$ and $\displaystyle\lim_{n\rightarrow +\infty} J_{s,\Phi}(u_n) =J_{s,\Phi}(u).$ 
		\end{itemize}
	\end{corollary}
	Now, we recall some definitions of operators of monotone type that we will use throughout this section.
	\begin{definition}
		Let $X$ be a reflexive Banach space with norm $\|\cdot\|_X$ and let $A:X \rightarrow X^\ast$ be an operator. Then $A$ is said to be
		\begin{itemize}
			\item[(i)] monotone (strictly monotone) if $\left\langle Au-Av,u-v\right\rangle\geq 0$ ($>0$), for all $u,v \in X$ with $u\neq v$;
			
			\item[(ii)] uniformly monotone if  $\left\langle Au-Av,u-v\right\rangle\geq \alpha(\|u-v\|)\|u-v\|$ for all $u,v \in X$, where $\alpha:[0,\infty) \rightarrow [0,\infty)$ is strictly increasing with $\alpha(0)=0$ and $\alpha(t)\rightarrow +\infty$, as $t\rightarrow \infty$;
			
			\item[(iii)] pseudomonotone if $u_n \rightharpoonup u$ weakly in $X$ and $\limsup_{n\rightarrow+\infty}\left\langle Au_n, u_n-u \right\rangle\leq 0 $ imply
			$$\left\langle Au, u-v\right\rangle \leq\liminf_{n\rightarrow +\infty} \left\langle Au_n, u_n -v \right\rangle, \quad \mbox{for all} \hspace{0,2cm} v\in X;$$
			
			\item[(iv)] coercive if there exists a function $\beta:[0,\infty)\rightarrow \mathbb{R}$ such that $\lim_{t\rightarrow +\infty}\beta(t)=+\infty$ and 
			$$\frac{\left\langle Au,u\right\rangle }{\|u\|_X}\geq \beta(\|u\|), \quad \mbox{for all}\hspace{0,2cm} u\in X.$$
		\end{itemize} 
	\end{definition}	
	
	\begin{lemma}\label{derivative}
		Let $s\in(0,1)$ and assume that $(\varphi_1)-(\varphi_3)$ hold. Then, $J_{s,\Phi}$ belongs to $C^1\big(W_0^{s,\Phi_{x,y}}(\Omega),\mathbb{R}\big)$ and its G\^{a}teaux derivative is given by 
		$$\langle J'_{s,\Phi}(u),v\rangle = \int_\Omega\int_\Omega \varphi_{x,y}\left(|D_su(x,y)|\right)  D_su(x,y)  D_sv(x,y)\;\mathrm{d}\mu,$$
		for all $u, v \in W_0^{s,\Phi_{x,y}}(\Omega)$.
	\end{lemma}
	\begin{proof}
		The proof is similar to \cite[Lemma 3.1]{Azroul1} and we omit here.
	\end{proof}	
	
	Next, we shall prove some monotonicity properties of the operator $J'_{s,\Phi}$.
	\begin{proposition}\label{monot-1}
		 The operator $J'_{s,\Phi}:W_0^{s,\Phi_{x,y}}(\Omega)\rightarrow \big(W_0^{s,\Phi_{x,y}}(\Omega)\big)^\ast$ satisfies the following properties:
		\begin{itemize}
			\item[(i)] $J'_{s,\Phi}$ is bounded, coercive and monotone;
			\item [(ii)] $J'_{s,\Phi}$ is pseudomonotone.
		\end{itemize}
	\end{proposition}
	
	\begin{proof}
		($i$)  Since $\Phi_{x,y}$ is convex, it follows that $J_{s,\Phi}$ is convex. Then, $J_{s,\Phi}'$ is a monotone operator. Next, we shall prove that $J'_{s,\Phi}$ is bounded. For this, let $u,v \in W_0^{s,\Phi_{x,y}}(\Omega)\setminus\{0\}$. It follows from Young's inequality \eqref{Young's type inequality}, \eqref{est-conjugate} and Lemma \ref{M-N2} $(i)$ that
		$$
		\begin{aligned}
		\left| \left\langle  J'_{s,\Phi}(u),\frac{v}{\|v\|} \right\rangle \right| & \leq  \int_{\Omega}\int_{\Omega} \varphi_{x,y}(|D_su|)|D_su|\left|\frac{D_sv}{\|v\|}\right| \mathrm{d}\mu \\
			& \leq   \int_{\Omega}\int_{\Omega} \left[ \widetilde{\Phi}_{x,y}\big(\varphi_{x,y}(|D_su|)|D_su|\big) + \Phi_{x,y}\left( \frac{|D_sv|}{\|v\|}\right) \right] \mathrm{d}\mu \\
			& \leq   \int_{\Omega}\int_{\Omega} \left[ 2^m \Phi_{x,y}\left(\|u\|\frac{|D_su|}{\|u\|}\right) + \Phi_{x,y}\left( \frac{|D_sv|}{\|v\|}\right) \right] \mathrm{d}\mu \\
			& \leq 2^m \xi_0^+(\|u\|) J'_{s,\Phi}\left(\frac{u}{\|u\|}\right) +  J'_{s,\Phi}\left(\frac{v}{\|v\|}\right)\\
			& \leq 2^m  \left( \xi_0^+(\|u\|) +1\right) .
		\end{aligned}
		$$
		Hence,
		$$ \|J'_{s,\Phi}(u)\|_\ast = \displaystyle\sup_{v \;\in W_0^{s,\Phi_{x,y}}(\Omega),\,v\neq 0} \frac{\left\langle  J'_{s,\Phi}(u),v \right\rangle }{\|v\|}  \leq 2^m \left( \xi_0^+(\|u\|) +1\right), $$
		which implies that $J'_{s,\Phi}$ is bounded. It remains to prove that $J'_{s,\Phi}$ is coercive. For each $u\in W_0^{s,\Phi_{x,y}}(\Omega)\setminus\{0\}$, it follows from condition $(\varphi_3)$ and Lemma \ref{M-N2} ($ii$) that
		$$
		\begin{aligned}
			\frac{\langle J'_{s,\Phi}(u), u\rangle}{\|u\|}&= \frac{1}{\|u\|} \int_{\Omega}\int_{\Omega} \varphi_{x,y}(|D_su|) (D_su)^2 \; \mathrm{d}\mu\\
			& \geq \frac{\ell}{\|u\|} \int_{\Omega}\int_{\Omega} \Phi_{x,y}\left(|D_su|\right) \mathrm{d}\mu\\
			& \geq \frac{\ell}{\|u\|}  \min\{\|u\|^\ell, \|u\|^m\}\\
			& = \ell \min\{\|u\|^{\ell-1}, \|u\|^{m-1}\}.
		\end{aligned}
		$$

		
		Hence, since $m\geq\ell>1$, we conclude that
		$$\lim_{\|u\|\rightarrow +\infty} \frac{\langle J'_{s,\Phi}(u), u\rangle}{\|u\|} = +\infty,$$
		which proves that $J'_{s,\Phi}(u)$ is coercive.
		
		($ii$) By Lemma \ref{derivative}, $J'_{s,\Phi}$ is continuous, in particular, is hemicontinuous. Thus, since $J'_{s,\Phi}$ is monotone, it follows from \cite[Proposition 27.6]{Zeidler} that $J'_{s,\Phi}$ is pseudomonotone.
	\end{proof}

		Now, let us assume the following conditions:
	\begin{itemize}
		\item [$(\varphi_4)$] for each $(x,y) \in \Omega\times\Omega$, $t\mapsto \varphi_{x,y}(t)$ is a $C^1$-function on $(0,+\infty)$;
		\item[$(\varphi_5)$] $t\mapsto \varphi_{x,y}(t)$ is increasing in $(0,+\infty)$
	\end{itemize}
	
	Under these conditions, we can state another monotonicity property of the operator $J'_{s,\Phi}$, which is motivated by the work of Montenegro \cite{Montenegro}.
	\begin{proposition}
		Suppose that $(\varphi_1), (\varphi_3), (\varphi_4)$ and $(\varphi_5)$ hold. Then, $J'_{s,\Phi}$ is uniformly monotone.
	\end{proposition}
	\begin{proof}
		Let $a_{x,y}(t)=\varphi_{x,y}(|t|)t$. In view of $(\varphi_4)$ and $(\varphi_5)$, we obtain
		\begin{equation}\label{est-3.1}
			a'_{x,y}(t)=\varphi'_{x,y}(|t|)\frac{t^2}{|t|}+\varphi_{x,y}(|t|)\geq \varphi_{x,y}(|t|), \quad \mbox{for all} \;t\neq0.
		\end{equation}
		For any $\xi,\eta\in \mathbb{R}$ and $0<t\leq \frac{1}{4}$ there holds
		$$\frac{1}{4}|\xi-\eta|\leq |t\xi +(1-t)\eta|.$$
		This fact combined with $(\varphi_3), (\varphi_5)$, \eqref{est-3.1} and Lemma \ref{M-N2} $(i)$, imply that
		$$
		\begin{aligned}
			\left(\varphi_{x,y}(|\xi|)\xi - \varphi_{x,y}(|\eta|)\eta\right)(\xi - \eta)&=\int_{0}^{1} \frac{\mathrm{d}}{\mathrm{d}t} \Big(a_{x,y}(t\xi +(1-t)\eta)\Big) (\xi-\eta)\;\mathrm{d}t\\
			& =\int_{0}^{1} a'_{x,y}(t\xi +(1-t)\eta) (\xi - \eta)^2 \; \mathrm{d}t\\
			& \geq \int_{0}^{1} \varphi_{x,y}(|t\xi +(1-t)\eta|) (\xi - \eta)^2 \; \mathrm{d}t\\
			& \geq \int_{0}^{\frac{1}{4}} \varphi_{x,y}(|t\xi +(1-t)\eta|) (\xi - \eta)^2 \;\mathrm{d}t\\
			& \geq \int_{0}^{\frac{1}{4}} 16 \varphi_{x,y}\left(\frac{1}{4}|\xi -\eta|\right) \left(\frac{1}{4}|\xi - \eta|\right)^2 \mathrm{d}t\\
			& \geq 4 \ell\Phi_{x,y}\left(\frac{1}{4}|\xi -\eta|\right)\\
			&\geq 4^{1-m}\ell\Phi_{x,y}\left(|\xi -\eta|\right).
		\end{aligned}
		$$
		Thus, using the above inequality and Lemma \ref{M-N2} $(ii)$, we have
		$$
		\begin{aligned}
			\left\langle J'_{s,\Phi}(u)- J'_{s,\Phi}(v), u-v\right\rangle&=\int_{\Omega} \int_{\Omega} \left(\varphi_{x,y}(|D_su|)D_su - \varphi_{x,y}(|D_sv|)D_sv\right) (D_su - D_sv)\;\mathrm{d}\mu\\
			& \geq  4^{1-m}\ell \int_{\Omega} \int_{\Omega} \Phi_{x,y}\left(| D_su -D_sv|\right)\mathrm{d}\mu\\
			& \geq 4^{1-m}\ell \min\{\|u-v\|^\ell, \|u-v\|^m\}\\
			& = 4^{1-m}\ell \min\{\|u-v\|^{\ell-1}, \|u-v\|^{m-1}\} \|u-v\|.
		\end{aligned}
		$$
		Therefore, considering the function $\alpha(t)= 4^{1-m}\ell\min\{t^{\ell-1}, t^{m-1}\}$ for $t\geq0$, we conclude that $J'_{s,\Phi}$ is uniformly monotone.
	\end{proof}

	\begin{definition}\label{S+definition}
		We say that $J_{s,\Phi}'$ satisfies the $(S_+)$-property if for a given $\{u_{n}\}_{n\in\mathbb{N}}\subset W_0^{s,\Phi_{x,y}}(\Omega)$ satisfying $u_n\rightharpoonup u$ weakly in $ W_0^{s,\Phi_{x,y}}(\Omega)$ and
		 $$\limsup_{n\rightarrow\infty}\langle J_{s,\Phi}'(u_n), u_n-u\rangle\leq 0,$$
		there holds $u_n\rightarrow u$ strongly in $ W_0^{s,\Phi_{x,y}}(\Omega)$.
	\end{definition}
	
	\begin{theorem}\label{(S+)-2.0}
		Assume that $(\varphi_1) - (\varphi_3)$ hold. Then, $J_{s,\Phi}'$ satisfies the $(S_+)$-property.
	\end{theorem}
	\begin{proof}
		Suppose that $u_n\rightharpoonup u$ weakly in $W_0^{s,\Phi_{x,y}}(\Omega)$ and $\limsup_{n\rightarrow\infty}\langle J_{s,\Phi}'(u_n),u_n-u\rangle\leq 0$. In order to prove that $u_n \rightarrow u$ strongly in $W_0^{s,\Phi_{x,y}}(\Omega)$, it is sufficient to show that
		\begin{equation}\label{lim-3.1}
			\lim_{n\rightarrow +\infty} J_{s,\Phi}(u_n-u) = \lim_{n\rightarrow +\infty} \int_{\Omega} \int_{\Omega} \Phi_{x,y}(|D_su_n - D_su|)\; \mathrm{d}\mu =0.
		\end{equation}  
		Since the embedding $W_0^{s,\Phi_{x,y}}(\Omega)\hookrightarrow L^1(\Omega)$ is compact (see Remark \ref{embed-comp-part}), we have that $u_n(x) \rightarrow u(x)$ a.e in $\Omega$. Then, $D_su_n(x,y) \rightarrow D_su(x,y)$ a.e. in $\Omega\times\Omega$, which implies that
		\begin{equation}\label{lim-3.2}
			\lim_{n\rightarrow +\infty}\Phi_{x,y}(|D_su_n(x,y) - D_su(x,y)|) |x-y|^{-N} = 0, \quad \mbox{a.e. in}\;\Omega\times\Omega.
		\end{equation}
		In view from \eqref{lim-3.2} and Vitali's Theorem \cite[Corollary 4.5.5]{Bogachev}, to prove \eqref{lim-3.1}, it is sufficient to prove that the sequence
		\begin{equation}\label{seq1}
			\left\lbrace\Phi_{x,y}(|D_su_n(x,y) - D_su(x,y)|) |x-y|^{-N}\right\rbrace_{n\in\mathbb{N}}
		\end{equation}
		has uniformly absolutely continuous integral over $\Omega\times \Omega$. Firstly, note that Lemma \ref{derivative} and the weak convergence $u_n\rightharpoonup u$ in $W_0^{s,\Phi_{x,y}}(\Omega)$ imply that
		$$\lim_{n\rightarrow+\infty}\left\langle  J_{s,\Phi}'(u), u_n - u\right\rangle=0.$$
		Thus,
		$$\limsup_{n\rightarrow+\infty}  \left\langle  J_{s,\Phi}'(u_n) -  J_{s,\Phi}'(u), u_n-u \right\rangle\leq 0.$$
		Hence, the monotonicity of operator $J_{s,\Phi}'$  jointly with the limit just above implies that
		$$0\leq\liminf_{n\rightarrow +\infty}\left\langle  J_{s,\Phi}'(u_n) - J_{s,\Phi}'(u), u_n - u\right\rangle\leq\limsup_{n\rightarrow+\infty}  \left\langle  J_{s,\Phi}'(u_n) -  J_{s,\Phi}'(u), u_n-u \right\rangle\leq 0,$$
		i.e.,
		\begin{equation}\label{lim-3.3}
			\lim_{n\rightarrow+\infty}  \left\langle  J_{s,\Phi}'(u_n) - J_{s,\Phi}'(u), u_n-u \right\rangle= 0.
		\end{equation} 
		For $n\in\mathbb{N}$, define
		$$f_n(x,y):=\Big(\varphi_{x,y}(|{D_s}u_n(x,y)|){D_s}u_n(x,y)-\varphi_{x,y}(|{D_s}u(x,y)|){D_s}u(x,y)\Big)\left({D_s}u_n(x,y)-{D_s}u(x,y)\right).$$
		Since $(\varphi_2)$ holds, a direct computation infers
		\begin{equation}\label{ms1}
			(t\varphi_{x,y}(|t|) - s\varphi_{x,y}(|s|)) (t-s)\geq0,\quad \mbox{for all} \; (x,y)\in \Omega\times \Omega \; \mbox{and} \; s,t \in \mathbb{R},
		\end{equation}
		see for instance \cite[Lemma 7.5]{Alves} or \cite[Proposition 2.5]{Da Silva}. The inequality \eqref{ms1} combined with the limit \eqref{lim-3.3} imply that the sequence $\{f_n(x,y)|x-y|^{-N}\}_{n\in\mathbb{N}}$ converges to $0$ in $L^1(\Omega\times \Omega)$.  Thus, by converse Vitali's Theorem \cite[Corollary 4.5.5]{Bogachev}, $\{f_n(x,y)|x-y|^{-N}\}_{n\in\mathbb{N}}$ has uniformly absolutely continuous integral over $\Omega\times\Omega$.
		
		Now, observe that 
		\begin{equation}\label{est-3.4}
			\begin{aligned}
				f_n(x,y)=&\varphi_{x,y}(|{D_s}u_n(x,y)|)(D_su_n)^2(x,y)+\varphi_{x,y}(|{D_s}u(x,y)|){(D_su)}^2(x,y)\\
				&-\varphi_{x,y}(|{D_s}u_n(x,y)|){D_s}u_n(x,y){D_s}u(x,y)-\varphi_{x,y}(|{D_s}u(x,y)|){D_s}u(x,y){D_s}u_n(x,y).\\
			\end{aligned}
		\end{equation} 
		For each $\varepsilon\in(0,1)$, using Young's inequality \eqref{Young's type inequality}, \eqref{est-conjugate}, \eqref{est-3.4}, Lemma \ref{M-N2} $(i)$ and $(\varphi_3)$, we obtain 
		$$
		\begin{aligned}
			\varphi_{x,y}(|{D_s}u_n(x,y)|)(D_su_n)^2(x,y) & = f_n(x,y) - \varphi_{x,y}(|{D_s}u(x,y)|)(D_su)^2(x,y)\\
			&\hspace{0,5cm}+ \varphi_{x,y}(|{D_s}u_n(x,y)|){D_s}u_n(x,y){D_s}u(x,y)\\
			& \hspace{0,5cm} +\varphi_{x,y}(|{D_s}u(x,y)|){D_s}u(x,y){D_s}u_n(x,y)\\
			&\leq f_n(x,y)+\varepsilon\widetilde\Phi_{x,y}(\varphi_{x,y}(|{D_s}u_n(x,y)|)|{D_s}u_n(x,y)|)\\
			&\hspace{0,5cm}+ C_\varepsilon\Phi_{x,y}(|{D_s}u(x,y)|)+ \widetilde{C}_\varepsilon\widetilde\Phi_{x,y}(\varphi_{x,y}(|{D_s}u(x,y)|)|{D_s}u(x,y)|)\\
			& \hspace{0,5cm}+\varepsilon\Phi_{x,y}(|{D_s}u_n(x,y)|)\\
			&\leq  f_n(x,y)+({C}_\varepsilon+2^m\widetilde{C}_\varepsilon )\Phi_{x,y}(|{D_s}u(x,y)|)\\
			&\hspace{0,5cm}+ \varepsilon (1+2^m)\ell^{-1}\varphi_{x,y}(|{D_s}u_n(x,y)|)({D_s}u_n)^2(x,y).
		\end{aligned}
		$$
		Thus, by choosing $0<\varepsilon<\frac{\ell}{1+2^m}$ sufficiently small and using $(\varphi_3)$, we obtain $C:=C(\varepsilon,\ell, m)>0$ such that
		\begin{equation}\label{est-3.5}
			\begin{aligned}
				{\Phi_{x,y}(|{D_s}u_n(x,y)|)}&\leq \ell^{-1}{\varphi_{x,y}(|{D_s}u_n(x,y)|)({D_s}u_n)^2(x,y)}\\
				&\leq C\left(f_n(x,y)+{\Phi_{x,y}(|{D_s}u(x,y)|)}\right).
			\end{aligned}
		\end{equation}
		Hence, using that $\Phi_{x,y}$ is convex, Lemma \ref{M-N2} $(i)$ and \eqref{est-3.5}, we obtain
		$$\begin{aligned}
			\Phi_{x,y}&\big(|{D_s}u_n(x,y)-{D_s}u(x,y)|\big)|x-y|^{-N}\leq \Phi_{x,y}\left(\frac{2|{D_s}u_n(x,y)|+2|{D_s}u(x,y)|}{2}\right)|x-y|^{-N}\\
			&\leq \left(2^{m-1}\Phi_{x,y}(|{D_s}u_n(x,y)|)+2^{m-1} \Phi_{x,y}(|{D_s}u(x,y)|)\right)|x-y|^{-N} \\
			&\leq 2^{m-1}C\left(f_n(x,y)+{\Phi_{x,y}(|{D_s}u(x,y)|)}\right)|x-y|^{-N} + 2^{m-1}{\Phi_{x,y}(|{D_s}u(x,y)|)}|x-y|^{-N},
		\end{aligned}$$
		which implies that the sequence \eqref{seq1} has uniformly absolutely continuous integral. Therefore, \eqref{lim-3.1} holds by Vitali's Theorem.
	\end{proof}
	
	\begin{remark}
		We point out that in Theorem \ref{(S+)-2.0} we can consider the condition $(\varphi_3)$ with $\ell\geq 1$ and therefore the space $W^{s,\Phi_{x,y}}(\Omega)$ is non-reflexive. Also, it is not used that $t\mapsto\Phi_{x,y}(\sqrt{t})$ is convex. Thus, Theorem \ref{(S+)-2.0} can be seen as a generalization of the result obtained by Bahrouni et al. \cite[Lemma 3.4]{S. Bahrouni et.al}.
	\end{remark}
	In the sequel, we will give an alternative proof for the $(S_+)$-property assuming a stronger hypothesis than $(\varphi_2)$, namely:
	\begin{itemize}
		\item[$(\varphi_2)'$] $t\mapsto t\varphi_{x,y}(t)$ is strictly increasing.
	\end{itemize}

		\begin{theorem}\label{S+}
			Assume that $(\varphi_1), (\varphi_2)'$ and $(\varphi_3)$ hold. Then, $J_{s,\Phi}'$ satisfies $(S_+)$-property.
		\end{theorem}
		\begin{proof}
			Using the same techniques as in the proof of Theorem \ref{(S+)-2.0}, we conclude that the sequence $\{f_{n}(x,y)\}_{n\in\mathbb{N}}$ defined by
			$$f_n(x,y):=\left(\varphi_{x,y}(|{D_s}u_n(x,y)|){D_s}u_n(x,y)-\varphi_{x,y}(|{D_s}u(x,y)|){D_s}u(x,y)\right)\left({{D_s}u_n(x,y)-{D_s}u(x,y)}\right)$$
			converges to $0$ in $L^1(\Omega\times \Omega,\mathrm{d}\mu)$. On the other hand, since $(\varphi_2)'$ holds, the inequality \eqref{ms1} becomes
			\begin{equation}\label{ms}
				(t\varphi_{x,y}(|t|) - s\varphi_{x,y}(|s|)) (t-s)>0,\quad \mbox{for all} \; (x,y)\in \Omega\times \Omega \; \mbox{and} \; t\neq s.
			\end{equation}
			This means that $t\mapsto t\varphi_{x,y}(|t|)$ is strictly monotone. Then, by \cite[Lemma 6]{DalMaso}, we conclude that ${D_s}u_n(x,y)\rightarrow {D_s}u(x,y),~\mu$-a.e. in $\Omega\times \Omega$. Moreover,  there exist $g\in L^1(\Omega\times \Omega,\mathrm{d}\mu)$ and a subsequence, still denoted by $\{f_n\}_{n\in\mathbb{N}}$, such that $|f_n|\leq g,~\mu$-a.e. in $\Omega\times\Omega$.
			Thus, by inequality \eqref{est-3.5}, we obtain $C:=C(\varepsilon,\ell,m)>0$ such that
			\begin{equation}\label{est-2}
				\begin{aligned}
					{\Phi_{x,y}(|{D_s}u_n(x,y)|)}&\leq\ell^{-1}{\varphi_{x,y}(|{D_s}u_n(x,y)|)({D_s}u_n)^2(x,y)}\\
					&\leq C\left(g(x,y)+{\Phi_{x,y}(|{D_s}u(x,y)|)}\right)\in L^1(\Omega\times\Omega,\mathrm{d}\mu).
				\end{aligned}
			\end{equation}
			By using the convexity of $\Phi_{x,y}$, Lemma \ref{M-N2} $(i)$ and \eqref{est-2}, we conclude that
			$$
			\begin{aligned}
				\Phi_{x,y}&\big(|{D_s}u_n(x,y)-{D_s}u(x,y)|\big)\leq{\Phi_{x,y}\left(\frac{2|{D_s}u_n(x,y)|+2|{D_s}u(x,y)|}{2}\right)}\\
				&\leq 2^{m-1}{\Phi_{x,y}(|{D_s}u_n(x,y)|)}+2^{m-1}{\Phi_{x,y}(|{D_s}u(x,y)|)}\\
				&\leq 2^{m-1}C\left(g(x,y)+{\Phi_{x,y}(|{D_s}u(x,y)|)}\right) + 2^{m-1}{\Phi_{x,y}(|{D_s}u(x,y)|)}\in L^1(\Omega\times\Omega,\mathrm{d}\mu).
			\end{aligned}
			$$
			Consequently, by Lebesgue's Dominated Convergence Theorem, there holds
			$$J_{s,\Phi}(u_n-u)=\int_\Omega\int_\Omega\Phi_{x,y}(|{D_s}u_n(x,y)-{D_s}u(x,y)|)\; d\mu\rightarrow 0.$$
			Therefore, by Proposition \ref{M-N2} $(ii)$, we have $\|u_n-u\|\rightarrow 0$, i.e., $u_n\rightarrow u$ strongly in $W_0^{s,\Phi_{x,y}}(\Omega)$. 
		\end{proof}
	
	The next result characterizes the strong convergence in the space $W_0^{s,\Phi_{x,y}}(\Omega)$ under the assumption $(\varphi_2)'$.
	\begin{proposition} 
		Assume that $(\varphi_1), (\varphi_2)'$ and $(\varphi_3)$ hold. Let $\{u_n\}_{n\in\mathbb{N}}$ be a sequence in $W_0^{s,\Phi_{x,y}}(\Omega)$. Then, $u_n \rightarrow u$ in $W_0^{s,\Phi_{x,y}}(\Omega)$ if and only if
		\begin{equation}\label{Cconv}
			\lim_{n\rightarrow +\infty} \left\langle J_{s,\Phi}'(u_n) - J_{s,\Phi}'(u), u_n - u\right\rangle = 0
		\end{equation}
	\end{proposition}
	\begin{proof}
		If $u_n \rightarrow u$, then by Proposition \ref{derivative} the limit \eqref{Cconv} holds. Conversely, assuming \eqref{Cconv} and arguing as in proof Theorem \ref{S+}, we obtain the desired result. 
	\end{proof}
	
	Next, we we introduce another monotonicity result in the presence of hypothesis $(\varphi_2)'$.
	\begin{proposition}\label{homeomorphism}
		Assume that $(\varphi_1), (\varphi_2)'$ and $(\varphi_3)$ hold. Then, $J'_{s,\Phi}$ is a homeomorphism strictly monotone.
	\end{proposition}
	\begin{proof}
		The strict monotonicity of $J'_{s,\Phi}$ follows from \eqref{ms}. Thus, by Proposition \ref{monot-1} ($i$) and Minty--Browder Theorem \cite[Theorem 26.A]{Zeidler}, $J'_{s,\Phi}$ is invertible and $(J'_{s,\Phi})^{-1}$ is strictly monotone and bounded. Therefore, in order to complete the proof of ($ii$) we only need to show that $(J'_{s,\Phi})^{-1}$ is continuous. For this purpose, let $\{g_n\}_{n\in\mathbb{N}}$ be a sequence such that $g_n \rightarrow g$ strongly in $\big(W_0^{s,\Phi_{x,y}}(\Omega)\big)^\ast$. By taking $u_n=(J'_{s,\Phi})^{-1}(g_n)$ and $u=(J'_{s,\Phi})^{-1}(g)$, it follows from the strong convergence of $\{g_n\}_{n\in\mathbb{N}}$ and the boundedness of $(J'_{s,\Phi})^{-1}$ that $\{u_n\}_{n\in\mathbb{N}}$ is bounded in $W_0^{s,\Phi_{x,y}}(\Omega)$. Thus, up to subsequence $u_n \rightharpoonup u_0$ in $W_0^{s,\Phi_{x,y}}(\Omega)$. Consequently,
		$$
		\begin{aligned}
			\lim_{n\rightarrow +\infty} \langle J'_{s,\Phi}(u_n)-J'_{s,\Phi}(u_0), u_n - u_0\rangle &= \lim_{n\rightarrow +\infty} \left( \langle J'_{s,\Phi}(u_n) - g, u_n - u_0\rangle + \langle g - J'_{s,\Phi}(u_0), u_n -u_0 \rangle \right)\\ 
			&= \lim_{n\rightarrow +\infty} \langle J'_{s,\Phi}(u_n) - g, u_n - u_0\rangle + \lim_{n\rightarrow +\infty} \langle g - J'_{s,\Phi}(u_0), u_n -u_0 \rangle\\
			& = 0.
		\end{aligned}
		$$
		This fact jointly with Theorem \ref{S+} imply that $u_n \rightarrow u_0$. Hence, by the continuity of the operator $J'_{s,\Phi}$ we obtain
		$$J'_{s,\Phi}(u_0)= \lim_{n\rightarrow +\infty} J'_{s,\Phi}(u_n) = \lim_{n\rightarrow +\infty} g_n = g = J'_{s,\Phi}(u),$$ 
		i.e., $u=u_0$. 
		Therefore, $(J'_{s,\Phi})^{-1}$ is continuous.
	\end{proof}
	
	\begin{remark}
		Let $\mathcal{J}_{s,\Phi}: W^{s,\Phi_{x,y}}(\Omega) \rightarrow \mathbb{R}$ be the modular function defined by
		$$\mathcal{J}_{s,\Phi}(u)= J_{s,\Phi}(u) + J_{\widehat{\Phi}}(u).$$
		In light of \cite[Proposition 2.1]{Azroul1}, 
		$$\|u\|_{\mathcal{J}_{s,\Phi}}:=\inf\left\{\lambda>0: \mathcal{J}_{s,\Phi}\left(\frac{u}{\lambda}\right)\leq 1\right\}$$
		is an equivalent norm on $W^{s,\Phi_{x,y}}(\Omega)$.
		Then, the above results still hold true if we replace $J_{s,\Phi}$ by $\mathcal{J}_{s,\Phi}$.
	\end{remark}
	
	\section{Application to a nonlocal fractional type problem}\label{Existence result}

	In this section, we investigate the existence of nontrivial solution for following class of fractional type problems 
	\begin{equation}\label{P}
		\left\{
		\begin{array}{ll}
			(-\Delta)_{\Phi_{x,y}}^s u = f(x,u),& \mbox{in }\Omega\\
			u=0,& \mbox{on }\mathbb{R}^N\setminus \Omega,
		\end{array}
		\right.
	\end{equation}
	where $N\geq 2$, $\Omega\subset\mathbb{R}^{N}$ is a bounded domain with Lipschitz boundary $\partial \Omega$, $(-\Delta)_{\Phi_{x,y}}^s$ is the so called fractional $(s,\Phi_{x,y})$-Laplacian operator defined by
	\begin{eqnarray}\label{a-frac}
		(-\Delta)_{\Phi_{x,y}}^s u(x)\displaystyle:= 2\lim_{\varepsilon\rightarrow 0} \int_{\mathbb{R}^N\setminus B_{\varepsilon}(x)} \varphi_{x,y}\left(\frac{|u(x)-u(y)|}{|x-y|^s}\right) \frac{u(x)-u(y)}{|x-y|^s}\frac{\mathrm{d}y}{|x-y|^{N+s}},
	\end{eqnarray}
	where $s \in (0,1)$,  $\Phi_{x,y}(t)=\int_0^{|t|} \tau \varphi_{x,y}(\tau)\,\mathrm{d}\tau$ is a generalized $N$-function and the nonlinearity $f:\Omega \times \mathbb{R} \rightarrow \mathbb{R}$ is a Carath\'{e}odory function that satisfies the following conditions:
	
	 \begin{itemize} 
	  \item[$(f_1)$] there exist a constant $C>0$ and $\Psi \in \mathcal{N}(\Omega)$ satisfying $\Psi \ll \widehat{\Phi}_s^\ast$  such that
	$$|f(x,t)|\leq C(1+ |t|\psi(x,|t|)), \quad x\in\Omega, \; t\in\mathbb{R},$$
	where $\Psi(x,t)=\displaystyle\int_0^{|t|} \tau\psi(x,\tau)\;\mathrm{d}\tau$ satisfies
	$$m<\ell_\Psi\leq \frac{t^2\psi(x,t)}{\Psi(x,t)}\leq m_\Psi<\infty, \quad x\in \Omega, \; t>0,$$
	and
	\begin{equation}\label{psi2}
		0<C_1\leq\Psi(x,1)\leq C_2<+\infty, \quad x\in\Omega,
	\end{equation}
for some constants $C_1$ and $C_2$.
	
	\item[$(f_2)$] there exists $\Gamma \in \mathcal{N}(\Omega)$ given by
	$\Gamma(x,t)=\int_0^{|t|} \tau\gamma(x,\tau)\; \mathrm{d}\tau$,
	where $\gamma:\Omega \times (0,+\infty) \rightarrow [0,+\infty)$ is a Carath\'{e}odory function which satisfies
	$${\frac{N}{\ell}<\ell_\Gamma}\leq\frac{t^2\gamma(x,t)}{\Gamma(x,t)}\leq m_\Gamma<\infty, \quad x\in \Omega, \; t>0$$
	and
	\begin{equation}\label{gamma2}
		\sup_{x\in\Omega} \Gamma(x,1)<\infty,
	\end{equation}
	such that
	$$\Gamma\left(x,\frac{F(x,t)}{|t|^\ell}\right)\leq C \overline{F}(x,t), \quad x\in\Omega, \;|t|\geq R,$$
	where $C, R$ are positive constant,
	$$F(x,t):=\int_0^t f(x,\tau)\; \mathrm{d}\tau\quad \mbox{and} \quad \overline{F}(x,t):= tf(x,t)- mF(x,t), \quad x \in \Omega, \; t\in\mathbb{R};$$
	
	\item[$(f_3)$] $\displaystyle\lim_{t\rightarrow +\infty} \frac{f(x,t)}{|t|^{m-1}}=\infty,$ uniformly for $x\in \Omega$;
	
	\item[$(f_4)$] $\displaystyle\lim_{t\rightarrow 0} \frac{f(x,t)}{|t|\widehat{\varphi}(x,|t|)} < \frac{1}{\lambda_1}$, uniformly for $x\in \Omega$, where $\lambda_1$ was introduced in \eqref{poincareconstant}.
	
\end{itemize}

\vspace{0,3cm}

Now, we list some remarks on our assumptions.

	\begin{remark}
		The assumption $(f_2)$ is a type of nonquadraticity condition at infinity, which was first introduced by Costa and Magalh\~{a}es, \cite{Costa and Magalhaes}, for the Laplace operator, i.e, when $\ell=m=2$. It was required that
		$$\liminf_{|t|\rightarrow +\infty}\frac{\overline{F}(x,t)}{|t|^\sigma}\geq a>0,$$
		holds for some $\sigma>0$. Such condition plays an important role in proving the compactness result, such as the Cerami condition.
	\end{remark}
	
	\begin{remark}
		If $\Psi$ and $\Gamma$ satisfy the conditions \eqref{psi2} and \eqref{gamma2} respectively, then by using $\Delta_2$-condition it is possible to prove that $\Psi$ and $\Gamma$ are bounded for every $k>0$, and consequently, $\widetilde{\Psi}(x,k)$ and $\widetilde{\Gamma}(x,k)$ are also bounded. 
	\end{remark}

	\begin{remark}
		The condition $m<\ell_\Gamma$ in $(f_2)$ implies that $\widehat{\Phi}\ll\Psi$. Indeed, by Lemmas \ref{M-N1} $(i)$ and \ref{M-N critical} $(i)$,
		$$\lim_{t\rightarrow+\infty} \frac{\widehat{\Phi}(x,kt)}{\Psi(x,t)}\leq \frac{\widehat{\Phi}(x,k)}{\Psi(x,1)}\lim_{t\rightarrow+\infty}\frac{t^m}{t^{\ell_\Gamma}} =0, \quad \mbox{uniformly for }x\in \Omega.$$
	\end{remark}
	
	\begin{remark}
		Note that assumption $(f_3)$ implies that
		$$\lim_{|t|\rightarrow +\infty} \frac{\ell F(x,t)}{|t|^\ell}=\lim_{|t|\rightarrow +\infty}\frac{f(x,t)}{|t|^{\ell -1}}=+\infty.$$
		Thus, in view of $(f_2)$ and fact $\Gamma \in \mathcal{N}(\Omega)$, it follows that
		$$\lim_{|t|\rightarrow +\infty}\Gamma\left(x, \frac{F(x,t)}{|t|^\ell}\right)=+\infty \quad \mbox{and} \quad \lim_{|t|\rightarrow +\infty} \overline{F}(x,t)=+\infty, \quad \mbox{for}\; x\in\Omega.$$
	\end{remark} 
	\vspace{0.5cm}

	In order to define the notion of weak solution for problem \eqref{P}, we need to require a symmetry assumption in $x$ and $y$ for the function $\varphi(x,y,t)$, precisely,
	\begin{equation*}
		\varphi(x,y,t)=\varphi(y,x,t), \quad \mbox{for all}\hspace{0.2cm} (x,y)\in\Omega\times\Omega \hspace{0.2cm} \mbox{and} \hspace{0.2cm} t\geq0.
	\end{equation*}
	\begin{definition}
		A function $u \in W_0^{\Phi_{x,y}}(\Omega)$ is said to be a weak solution for Problem \eqref{P} if satisfies
		$$\int_\Omega\int_\Omega \varphi_{x,y}\left(|D_su(x,y)|\right)  D_su(x,y) D_sv(x,y) \;\mathrm{d}\mu =\int_{\Omega} f(x,u)v \;\mathrm{d}x, $$
		for all $v \in W_0^{\Phi_{x,y}}(\Omega)$.
	\end{definition}

The following result is an immediate consequence of the Proposition \ref{homeomorphism}.
\begin{proposition}
	Assume that $(\varphi_1), (\varphi_2)'$ and $(\varphi_3)$ hold. If $f(x,t)=f(x)$ and $f\in L_{\widetilde{(\widehat{\Phi}_{s}^\ast)}_x}(\Omega)$, then Problem \eqref{P} has a unique weak solution.
\end{proposition}
\begin{proof}
	By H\"{o}lder's inequality and Proposition \ref{imbed-critical-bounded} $(i)$, one can see that
	$$\langle f, v\rangle:= \int_{\Omega} f(x)v(x) \; \mathrm{d}x, \quad v\in W_0^{s,\Phi_{x,y}}(\Omega),$$ 
	defines a continuous linear functional on $W_0^{\Phi_{x,y}}(\Omega)$, i.e., $f \in \big(W_0^{s,\Phi_{x,y} }(\Omega)\big)^\ast$. Since $J_{s,\Phi}'$ is bijective by Proposition \ref{homeomorphism}, it follows that \eqref{P} has a unique weak solution.
\end{proof}

Our main existence result can be stated as follows.
\begin{theorem}\label{theorem1}
	Assume that $(\varphi_1) - (\varphi_3)$ and $(f_1) - (f_4)$ hold. Then, Problem \eqref{P} has at least one nontrivial weak solution. 
\end{theorem}	
	
	Associated to Problem \eqref{P}, we introduce the energy functional $I_{s,\Phi}:W_0^{s,\Phi_{x,y}}(\Omega)\rightarrow\mathbb{R}$ defined by
	$$I_{s,\Phi}(u)=J_{s,\Phi}(u) -\int_{\Omega}F(x,u)\;\mathrm{d}x.$$
	In view of assumption $(f_1)$ and Lemma \ref{derivative}, one can show that $I_{s,\Phi}$ is well-defined, belongs to $C^1$ and it's G\^{a}teaux derivative is given by
	$$\left\langle I_{s,\Phi}'(u), v\right\rangle = \left\langle J'_{s,\Phi}(u),v\right\rangle  - \int_{\Omega} f(x,u)v \mathrm{d}x, $$
	see for instance \cite[Lemmas 3.1 and 3.2]{Azroul1}. Furthermore, critical points of $I_{s,\Phi}$ are solutions to Problem \eqref{P}, and conversely.
	
	In order to apply variational methods, we recall that the functional $I_{s,\Phi}$ is said to satisfy $(C)_c$-condition if, for $c\in \mathbb{R}$, any sequence $\{u_n\}_{n \in\mathbb{N}}$ in $W_0^{\Phi_{x,y}}(\Omega)$ such that
	$$I_{s,\Phi}(u_n) \rightarrow c \quad \mbox{and} \quad (1+\|u\|)\|I_{s,\Phi}'(u_n)\|_\ast \rightarrow 0$$
	admits a convergent subsequence. A sequence satisfying the above condition is called $(C)_c$-sequence. We say that $I_{s,\Phi}$ satisfies the Cerami condition, in short $(C)$-condition, if it satisfies $(C)_c$-condition for all $c \in \mathbb{R}$.
	
	We apply the following variant of the well-known mountain pass theorem.
	\begin{lemma}(Mountain Pass Theorem \cite[Theorem 5.4.6]{Papageorgiou-Radulescu-Repovs})\label{Mountain pass}
		Let $X$ be a Banach space and suppose that $I\in C^{1}(x, \mathbb{R})$,
		$$\max\{I(u_0),I(u_1)\}\leq\inf_{\|u-u_0\|=\rho} I(u)=:\alpha_\rho,$$
		for some $u_0, u_1 \in X$ and $\|u_1-u_0\|>\rho>0$. If $I$ satisfies the $(C)_c$-condition at
		$$c=\inf_{\gamma\in\varGamma}\max_{s\in[0,1]} I_{s,\Phi}(\gamma(s)),$$
		where $\varGamma=\{\gamma\in C([0,1], X):\; \gamma(0)=u_0\; \mbox{and}\; \gamma(1)=u_1\}$, then $c \geq \alpha_\rho$ and $c$ is a critical value of  $I$.
	\end{lemma}

	Next, we show that $I_{s,\Phi}$ possesses the mountain pass geometry.
	\begin{lemma}\label{mountain pass geometry}
		Assume that $(f_1)$, $(f_3)$ and $(f_4)$ hold. Then, $I_{s,\Phi}$ satisfies the following assertions:
		\begin{itemize}
			\item[(i)] there exist $\rho, \alpha > 0$ such that $I_{s,\Phi}(u) \geq \alpha$, for all $u \in W_0^{s,\Phi_{x,y}}(\Omega)$ such that $\|u\|=\rho$;
			
			\item[(ii)] there exists $e \in W_0^{s,\Phi_{x,y}}(\Omega)$ with $\|e\|>\rho$ such that $I_{s,\Phi}(e) < 0$.
		\end{itemize}
	\end{lemma}
	\begin{proof}
		$(i)$ Arguing as in \cite[Lemma 3.1]{Carvalho et.al}, if $(f_1)$ and $(f_4)$ hold, we derive that
		\begin{equation}\label{est4.1.1}
			F(x,t)\leq \left(\frac{1 - \varepsilon}{\lambda_1}\right)\widehat{\Phi}(x,|t|) + C\Psi(x,|t|), \quad x\in\Omega, \; t\in\mathbb{R},
		\end{equation}
		for $\varepsilon \in(0,1)$ sufficiently small. Thus, by using \eqref{poincareconstant}, \eqref{est4.1.1}, the embedding $W_0^{\Phi_{x,y}}(\Omega)\hookrightarrow L^{\Psi_x}(\Omega)$ and Lemmas \ref{M-N1} $(ii)$ and \ref{M-N2} $(ii)$, we get
		$$
		\begin{aligned}
			I_{s,\Phi}(u) &= J_{s,\Phi}(u) - \int_{\Omega} F(x,u)\mathrm{d}x \\
			& = J_{s,\Phi}(u) -\left(\frac{1- \varepsilon}{\lambda_1} \right)\int_{\Omega}\widehat{\Phi}(x,|u|)\;\mathrm{d}x-C\int_{\Omega}\Psi(x,|u|)\;\mathrm{d}x\\
			& \geq J_{s,\Phi}(u) - (1-\varepsilon)J_{s,\Phi}(u) - C\int_{\Omega}\Psi(x,|u|)\;\mathrm{d}x\\
			& \geq \varepsilon \min\left\lbrace\|u\|^{\ell},\|u\|^{m}\right\rbrace - C\max\left\lbrace\|u\|^{\ell_\Psi}, \|u\|^{m_\Psi}\right\rbrace.
		\end{aligned}
		$$
		Hence, by taking $\|u\|\leq1$ and using that $m<\ell_\Psi$, we have
		$$I_{s,\Phi}(u) \geq \|u\|^m\left(\varepsilon - C\|u\|^{\ell_\Psi - m}\right) \geq \frac{\rho^m \varepsilon}{2}=:\alpha>0,$$
		whenever
		$$\|u\|=\rho:=\min\left\lbrace 1, \left(\frac{\varepsilon}{2C}\right)^{\frac{1}{\ell_\Psi-m}}\right\rbrace, $$
		which proves that $(i)$ is satisfied.
		
		$(ii)$ It is not hard to show that, if $(f_3)$ holds, then for each $M > 0$, there exists $C>0$ such that
		$$F(x,t)\geq M|t|^m - C,\quad \forall x \in \Omega, \; \forall t\in\mathbb{R}.$$
		Let $u_0\in W_0^{s,\Phi_{x,y}}(\Omega)$ be such that $\|u_0\|=1$. In view of the above relation and Lemma \ref{M-N2} ($ii$), for  $t\geq1$ large enough, we have that
		$$
		\begin{aligned}
			I_{s,\Phi}(tu_0)&= J_{s,\Phi}(tu_0) - \int_{\Omega} F(x,tu_0)\,\mathrm{d}x\\
			& \leq t^m - Mt^m\int_{\Omega}|u_0|^m\;\mathrm{d}x+ C|\Omega|\\
			& =\left(1-M\int_{\Omega}|u_0|^m\;\mathrm{d}x\right) t^m+ C|\Omega|.
		\end{aligned}
		$$
		By choosing $M>\displaystyle\left(\int_{\Omega}|u_0|^m \mathrm{d}x\right)^{-1}$, it follows that $I_{s,\Phi}(tu_0) \rightarrow -\infty$ when $t \rightarrow +\infty$. Hence, the assertion ($ii$) follows by taking $e = t u_0$ with $t$ sufficiently large.
	\end{proof}
	
	The next step consists in proving that the functional $I_{s,\Phi}$ satisfies $(C)$-condition. For this, we first show that any Cerami sequence for $I_{s,\Phi}$ is bounded.
	\begin{lemma}\label{boundeness-Cerami-sequence}
		Assume that $(f_1)-(f_4)$ hold. Then, any Cerami sequence for the functional $I_{s,\Phi}$ is bounded.
	\end{lemma}
	
	\begin{proof}
		For any $c\in \mathbb{R}$, let $\{u_n\}_{n\in\mathbb{N}}$ be a $(C)_c$-sequence. Assume by contradiction that $\{u_n\}_{n\in\mathbb{N}}$  is unbounded in $W_0^{s,\Phi_{x,y}}(\Omega)$. Then, there exists a subsequence, still labeled $\{u_n\}_{n\in\mathbb{N}}$, such that
		$$\|u_n\|\rightarrow +\infty, \quad I_{s,\Phi}(u_n) \rightarrow c \quad \mbox{and}\quad (1+\|u_n\|)\|I_{s,\Phi}'(u_n)\|_\ast\rightarrow 0.$$
		In view of assumption $(\varphi_3)$ we can see that
		\begin{equation}\label{est-4.1}	
			\begin{aligned}
				mI_{s,\Phi}(u_n) - \langle I_{s,\Phi}'(u_n), u_n \rangle &= \int_{\Omega}\int_{\Omega} \left( m\Phi_{x,y}(|D_su_n|)- \varphi_{x,y}(|D_su_n|)|D_su_n|^2\right) \mathrm{d}\mu\\ 
				& \quad+ \int_{\Omega}\left( f(x,u_n)u_n - m F(x,u_n)\right)\mathrm{d}x\\
				& \geq \int_{\Omega} \overline{F}(x,u_n)\;\mathrm{d}x.
			\end{aligned}
		\end{equation}
		Let $v_n:=\frac{u_n}{\|u_n\|}$. Since $\{v_n\}_{n\in\mathbb{N}}$ is bounded in $W_0^{s,\Phi_{x,y}}(\Omega)$, by reflexivity and the compact embedding $W_0^{s,\Phi_{x,y}}(\Omega)\hookrightarrow L^\Psi(\Omega)$,  there exists a subsequence, still denoted by $\{v_n\}_{n\in\mathbb{N}}$, such that 
		$$\left\{\begin{aligned}
			& v_n \rightharpoonup v, \quad \mbox{in} \quad W_0^{s,\Phi_{x,y}}(\Omega),\\
			&  v_n \rightarrow v, \quad \mbox{in} \quad L^\Psi(\Omega),\\
			& v_n(x) \rightarrow v(x), \quad \mbox{for a.e} \quad x\in \Omega.
		\end{aligned}\right.$$
		We split the proof into two cases. 
		
		\vspace{0,3cm}
		
		\noindent {\bf Case 1:} The set $\{v\neq 0\}:=\{x\in \Omega: v(x)\neq 0\}$ has positive Lebesgue measure. 
		
		\vspace{0,3cm}
		
		In this case, we have
		$$|u_n|=|v_n|\|u_n\| \rightarrow +\infty, \quad \mbox{in} \quad \{v\neq0\}.$$
		Since $\overline{F}(x,t)$ is positive, it follows from $(f_2)$, $(f_{3})$,  Fatou's Lemma and \eqref{est-4.1} that
		$$mc\geq \liminf_{n\rightarrow +\infty} \int_{\Omega}\overline{F}(x,u_n) \;\mathrm{d}x\geq\int_{\Omega}\liminf_{n\rightarrow +\infty} \overline{F}(x,u_n) \;\mathrm{d}x=+\infty,$$
		which is an absurd. This completes the proof for Case $1$.
		
		\vspace{0,3cm}
		
		\noindent {\bf Case 2:} $ v(x) = 0$, for a.e. $x\in\Omega$. 
		
		\vspace{0,3cm}
		
		It follows from the definition of $I_{s,\Phi}$, Lemma \ref{M-N2} $(ii)$ and $\|u_n\|\rightarrow +\infty$ that
		$$\|u_n\|^\ell \leq I_{s,\Phi}(u_n) + \int_{\Omega} F(x,u_n)\;\mathrm{d}x,$$
		wich implies that
		\begin{equation}\label{est-4.3}
			1\leq \frac{I_{s,\Phi}(u_n)}{\|u_n\|^\ell} +\int_{\{|u_n|\leq R\}} \frac{ F(x,u_n)}{\|u_n\|^\ell}\; \mathrm{d}x + \int_{\{|u_n|>R\} }\frac{ F(x,u_n)}{\|u_n\|^\ell} \;\mathrm{d}x.
		\end{equation}
		Next, we show that all terms of the right-hand side of \eqref{est-4.3} tend to zero when $n\rightarrow +\infty$, which is the desired contradiction. Indeed, since $\{I_{s,\Phi}(u_n)\}_{n\in\mathbb{N}}$ is bounded, the first term obviously tends to zero. For the second one, it follows from $(f_1)$ that for any $R>0$, there holds
		$$|F(x,t)|\leq \int_0^t|f(x,\tau)|\leq C(|t|+\Psi(x,|t|))\leq C(R), \quad x\in\Omega,\; |t|\leq R.$$
		Thus, 
		\begin{equation}\label{lim-1}
			\int_{\{|u_n|\leq R\}} \frac{F(x,u_n)}{\|u_n\|^\ell}\;\mathrm{d}x\leq \frac{C(R)|\Omega|}{\|u_n\|^\ell}=o_n(1).
		\end{equation}
		For the third term, by H\"{o}lder's inequality we have
		\begin{equation}\label{est-4.4}
			\begin{aligned}
				\int_{\{|u_n|>R\} }\frac{ F(x,u_n)}{\|u_n\|^\ell} \;\mathrm{d}x &=\int_{\{|u_n|\leq R\}} \frac{F(x,u_n)}{|u_n|^\ell} |v_n|^\ell\; \mathrm{d}x\\
				& \leq 2 \left\| \frac{F(x,u_n)}{|u_n|^\ell}\chi_{\{|u_n|>R\}}\right\|_\Gamma \left\||v_n|^\ell \chi_{\{|u_n|>R\}}\right\|_{\widetilde{\Gamma}}.
			\end{aligned}
		\end{equation}
		In view of Lemma \ref{M-N1} $(ii)$ and $(f_2)$, there exists a constant $C:=C(\ell_\Gamma, m_\Gamma)>0$ such that
		\begin{equation}\label{est-4.5}
			\left\| \frac{F(x,u_n)}{|u_n|^\ell}\chi_{\{|u_n|>R\}}\right\|_\Gamma \leq C \max\left\lbrace \left( \int_{\Omega}\overline{F}(x,u_n)\;\mathrm{d}x\right)^{\frac{1}{\ell_\Gamma}}, \left( \int_{\Omega}\overline{F}(x,u_n)\;\mathrm{d}x\right)^{\frac{1}{m_\Gamma}} \right\rbrace + C.
		\end{equation}
		By using \eqref{est-4.1}, we also have that
		\begin{equation}\label{est-4.6}
			\int_{\Omega} \overline{F}(x,u_n)\;\mathrm{d}x\leq mI_{s,\Phi}(u_n) - \langle I_{s,\Phi}'(u_n), u_n \rangle \leq C.
		\end{equation}
		Therefore, by combing \eqref{est-4.5} and \eqref{est-4.6}, we obtain that $\left\lbrace \left\| \frac{F(x,u_n)}{|u_n|^\ell}\chi_{\{|u_n|>R\}}\right\|_\Gamma \right\rbrace_{n\in\mathbb{N}}$ is bounded.
		Now, we shall prove that
		\begin{equation}\label{lim-2}
			\lim_{n\rightarrow +\infty} \left\||v_n|^\ell \chi_{\{|u_n|>R\}}\right\|_{\widetilde{\Gamma}}=0.
		\end{equation}
		Indeed, let $H(x,t):=\widetilde{\Gamma}(x,|t|^\ell)$. Since $\ell>1$ and $\widetilde{\Gamma} \in \mathcal{N}(\Omega)$, it follows that $H\in \mathcal{N}(\Omega)$. In addition, since \eqref{gamma2} and $\frac{N}{\ell}<\ell_\Gamma$ hold, we have that $H(x,k)$ is bounded for each $k>0$ and
		$$\ell\widetilde{\ell_\Gamma}<\frac{N\ell}{N-\ell} < \frac{N\ell}{N-s\ell} =\ell_s^\ast.$$ 
		Then, for each $k>0$, it follows from the last inequality, Lemma \ref{M-N1} $(iii)$ and Lemma \ref{M-N critical} $(i)$ that
		$$\lim_{t\rightarrow+\infty} \frac{H(x,tk)}{\widehat{\Phi}_s^\ast(x,t)}\leq \frac{H(x,k)}{\widehat{\Phi}_s^\ast(x,1)}\lim_{t\rightarrow+\infty}\frac{t^{\ell \widetilde{\ell_\Gamma}}}{t^{\ell_s^\ast}} =0, \quad \mbox{uniformly for }x\in \Omega,$$
		which shows that $H\ll\widehat{\Phi}_s^\ast$. Hence, by Proposition \ref{imbed-critical-bounded} $(ii)$ the embedding $W_0^{s,\Phi_{x,y}}(\Omega)\hookrightarrow L^{H_x}(\Omega)$ is compact, which implies that
		$$\lim_{n\rightarrow +\infty} \int_{\Omega} \widetilde{\Gamma}(x,|v_n|^\ell)\; \mathrm{d}x = 0.$$
		Thus, the limit \eqref{lim-2} holds. Since  $\left\lbrace \left\| \frac{F(x,u_n)}{|u_n|^\ell}\chi_{\{|u_n|>R\}}\right\|_\Gamma \right\rbrace_{n\in\mathbb{N}}$ is bounded, we conclude from \eqref{est-4.4} and \eqref{lim-2} that
		\begin{equation}\label{lim-3}
			\int_{\{|u_n|>R\} }\frac{ F(x,u_n)}{\|u_n\|^\ell} \;\mathrm{d}x=o_n(1),
		\end{equation}
		which finishes the proof.
	\end{proof}
	
	\begin{lemma}\label{cerami condition}
		Assume that the assumptions $(f_1)-(f_4)$ hold. Then, $I_{s,\Phi}$ satisfies the Cerami condition.
	\end{lemma}
	\begin{proof}
		Let $\{u_n\}_{n\in\mathbb{N}}$ be a Cerami sequence for $I_{s,\Phi}$. It follows from Lemma \ref{boundeness-Cerami-sequence} and the reflexivity of $W_0^{\Phi_{x,y}}(\Omega)$ that, up to subsequence, there exists $u \in W_0^{\Phi_{x,y}}(\Omega)$ such that $u_n\rightharpoonup u$, weakly in $W_0^{\Phi_{x,y}}(\Omega)$. We claim that $u_n \rightarrow u$ strongly in $W_0^{\Phi_{x,y}}(\Omega)$. By Proposition \ref{imbed-critical-bounded} ($ii$), the embedding $W_0^{s,\Phi_{x,y}}(\Omega)\hookrightarrow L^\Psi(\Omega)$ is compact, so $u_n \rightarrow u$ in $L^\Psi(\Omega)$. Furthermore, since $(1+\|u_n\|)\|I_{s,\Phi}'(u_n)\|_\ast \rightarrow 0$, we get
		\begin{equation}\label{est-4.7}
			\langle J'_{s,\Phi}(u_n), u_n - u \rangle=o_n(1) + \int_{\Omega}f(x,u_n)(u_n - u)\; \mathrm{d}x.
		\end{equation}
		On the other hand, the convexity of $\widetilde{\Psi}(x,\cdot)$ combined with $(f_1)$, Lemma \ref{M-N1} $(iii)$, the relation $\widetilde{\Psi}(x, t\psi(x,t))\leq \Psi(x,2t)$ and $\Delta_2$-condition imply that
		$$\begin{aligned}
			\widetilde{\Psi}(x,|f(x,u_n)|)& \leq \widetilde{\Psi} (x, C|u_n| +C|u_n|\psi(x,|u_n|))\\
			&\leq 2^{-1}\widetilde{\Psi}(x,2C)+2^{-1} \widetilde{\Psi}\big(x,2C|u_n|\psi(x,|u_n|)\big)\\
			& \leq C_1+C _2\widetilde{\Psi}(x,|u_n|\psi(x,|u_n|)) \\
			& \leq C_1+ 2^{m_\Psi}C_2\Psi(x,|u_n|),
		\end{aligned}$$
		which implies that sequence $\{f(\cdot,u_n(\cdot))\}_{n\in \mathbb{N}}$ is bounded in $L^{\widetilde{\Psi}}(\Omega)$. Thus, combining H\"{o}lder's inequality jointly with the convergence $u_n \rightarrow u$ in $L^\Psi(\Omega)$, we obtain
		\begin{equation}\label{est-4.8}
			\left| \int_{\Omega}f(x,u_n)(u_n - u)\; \mathrm{d}x \right| \leq 2\|f(\cdot,u_n)\|_{\widetilde{\Psi}} \|u_n - u\|_{\Psi}\rightarrow 0.
		\end{equation}
		Hence, from \eqref{est-4.7} and \eqref{est-4.8} it follows that
		$$\limsup_{n\rightarrow+\infty} \langle J_{s,\Phi}'(u_n), u_n - u\rangle = 0.$$
		Therefore, applying Theorem \ref{(S+)-2.0}, we conclude that $u_n\rightarrow u$.
	\end{proof}
	
	\noindent {\bf Proof of Theorem \ref{theorem1}.}  In view of Lemmas \ref{mountain pass geometry}, \ref{boundeness-Cerami-sequence} and \ref{cerami condition}, we can apply the Mountain Pass Theorem (Lemma \ref{Mountain pass}) to deduce that the functional $I_{s,\Phi}$ has a critical value $c\geq\alpha>0$ given by
	$$c=\inf_{\gamma\in\varGamma}\max_{s\in[0,1]} I_{s,\Phi}(\gamma(s)),$$
	where $\varGamma=\{\gamma\in C([0,1], W_0^{s,\Phi_{x,y}}(\Omega)):\; \gamma(0)=0\; \mbox{and}\; \gamma(1)=e\}$ and $e\in  W_0^{s,\Phi_{x,y}}(\Omega)$ is defined in Lemma \ref{mountain pass geometry}. Therefore, Problem \eqref{P} has a weak solution nontrivial. $\square$
	
	\section{Some classes of problems}\label{examples}

	In this section, we present some examples of functions $\Phi_{x,y}$ and source functions $f$ for which the existence result Theorem \ref{theorem1} may be applied. 

	Given $s\in(0,1)$ and a generalized $N$-function $\Phi_{x,y}$, we consider the following problem
	\begin{equation}\label{eqL}
	\left\{
	\begin{array}{ll}
	(-\Delta)^s_{\Phi_{x,y}} u =f(x,u),& \mbox{in }\Omega,\\
	u=0,& \mbox{on }\mathbb{R}^N\setminus \Omega,
	\end{array}
	\right.
	\end{equation}
where $f(x,t)=\frac{d}{dt}F(x,t)$ satisfies some suitable structural properties.

\subsection{Double phase problem}\label{ex1}

For $1<p<q<N$ and $a\in L^\infty(\Omega\times \Omega)$ a non-negative symmetric function, we consider the generalized $N$-function given by
$$
\Phi_{x,y}(t)=\frac{t^p}{p} + a(x,y) \frac{t^q}{q}.
$$
This gives the operator
$$
(-\Delta)^s_{\Phi_{x,y}}  u := (-\Delta)_p^s u +  (-\Delta)_{q,a}^s u,
$$
where (up to multiplicative constant) $(-\Delta)_p^s$ is the so called fractional $p$-Laplacian operator and $(-\Delta)_{q,a}^s$ is the anisotropic fractional $p$-Laplacian defined as
		\begin{eqnarray}\label{a-frac}
			(-\Delta)_{q,a}^s u(x)\displaystyle:= 2\lim_{\varepsilon\rightarrow 0} \int_{\mathbb{R}^N\setminus B_{\varepsilon}(x)} a(x,y)\frac{|u(x)-u(y)|^{q-2}(u(x)- u(y))}{|x-y|^{N+sq}}\;\mathrm{d}y.
		\end{eqnarray}
In this case, the nonlocal operator present in \eqref{eqL} is associated with the   energy functional
		\begin{equation}\label{func-double-phase}
			J_{s,\Phi}(u)=\int_{\Omega\times\Omega}\Phi_{x,y}(D_{s}u)\;\mathrm{d}\mu=\int_{\Omega\times\Omega} \left( \frac{|D_su|^p}{p} +a(x,y)\frac{|D_su|^q}{q}\right) \mathrm{d}\mu,\quad u\in W_0^{s,\Phi_{x,y}}(\Omega),
		\end{equation}
		whose integrand $\Phi_{x,y}(t)$ shows an unbalanced growth, precisely
		$$C_1|t|^p\leq\Phi_{x,y}(t)\leq C_2(|t|^p+|t|^q), \;\; \mbox{for a.e.}\; (x,y)\in \Omega\times\Omega \;\mbox{and for all}\; t\in\mathbb{R},$$ 
		with $C_1,C_2>0$.  The main feature of the functional integral \eqref{func-double-phase} is the change of ellipticity and growth properties on the set where the weight function $a(\cdot,\cdot)$ vanishes. More precisely, the energy density of $J_{s,\Phi}$ is controlled by $D_su$ at an order $q$ in the set $\{(x,y)\in \Omega\times\Omega: \; a(x,y)\neq 0\}$ and at an order $p$ in the set $\{(x,y)\in \Omega\times\Omega: \; a(x,y)= 0\}$. For this reason, this operator is known as fractional double phase operator, which is a special case of problems with non-standard growth.  Due to the unbalanced growth of $\Phi_{x,y}(\cdot)$, the classical fractional Sobolev space is not suitable to analize \eqref{eqL} and so we have to use the general abstract setting of the new fractional Musielak-Sobolev spaces.

		Fractional double phase problems are motivated by numerous local and nonlocal models arising in many fields of mathematical physics, for instance, composite materials, fractional quantum mechanics in the study of particles on stochastic fields, fractional superdiffusion, fractional white-noise limit and several equations that appear in the electromagnetism, electrostatics and electrodynamics as a model based on a modification of Maxwell’s Lagrangian density. For more details, see \cite{Ambrosio-Radulescu1, Zhang-Tang-Radulescu} and the references therein.

		We point out that, in the local case $s=1$, Zhikov \cite{Zhikov 1} was the first to investigate the integral functional of the double phase type in the context of the theory of elasticity and calculus of variations to describe models of strongly anisotropic materials. For instance, in the elasticity theory, the weight function $a(\cdot,\cdot)$ is called modulating coefficient and dictates the geometry of composites made of two different materials with distinct power hardening $p$ and $q$, see Zhikov \cite{Zhikov 2}.

		When $a\equiv1$ or $a\equiv0$, Problem \eqref{eqL} involves the well-known fractional  $(p,q)$-Laplacian operator or the fractional $p$-Laplacian operator, respectively. As far as we know, this is the first work that addresses the fractional version of the classic double-phase problem (e.g., see \cite{Liu and Dai 1,Colasuonno and Squassina}) in which the weight function is not necessarily a constant and can vanishes. 
		
		\vspace{0,2cm}
		Coming back to problem \eqref{eqL},  it is clear that $\Phi_{x,y}(1)\leq C$ for a.e. $(x,y) \in\Omega\times\Omega$ and
		$$\Phi_{x,y}(2t)=|2t|^p+a(x,y)|2t|^q\leq 2^q\Phi(x,y,t),\quad (x,y) \in \Omega \times \Omega,\; t\in\mathbb{R},$$
		that is, $\Phi$ satisfies the $(\Delta_{2})$-condition. We note that by direct computation,
		$$t\varphi_{x,y}(t)=\Phi_{x,y}'(t)=|t|^{p-2}t + a(x,y)|t|^{q-2}t, \; t\neq0,$$
		and thus, the conditions $(\varphi_1)-(\varphi_3)$ are satisfied with $\ell=p$ and $m=q$. 
		
		Regarding to the nonlinearity we may consider $F(x,t)=|t|^r\log(1+|t|)$ with $q<r<p_s^\ast -1$. In this case we have that
		$$f(x,t)=r|t|^{r-2}t\log(1+|t|)+\frac{|t|^{r-1}t}{1+|t|}, \; t\neq0, \qquad \overline{F}(x,t)=\frac{|t|^{r+1}}{1+|t|} +(r-q)|t|^r\log(1+|t|).
		$$
		Let us see that $f$ satisfies the conditions $(f_1)-(f_4)$. 
		Considering $\Psi(x,t)=F(x,t)$, it follows that
		$$t\psi(x,t)=\Psi'(x,t)=rt^{r-1}\log(1+t)\frac{t^r}{1+t}, \quad t>0,$$
		and so,
		$$\frac{t^2\psi(x,t)}{\Psi(x,t)}=r+\frac{t}{(1+t)\log(1+t)}, \quad t>0.$$
		In this case, taking $\ell_\Psi=r$, $m_\Psi=r+1$, we have that $(f_1)$ is satisfied.
		
		Now, consider the function $\Gamma(x,t)=|t|^\gamma$ with $N/p<\gamma<r/(r-p)$. Thus,
		$$\Gamma\left(x,\frac{F(x,t)}{|t|^p}\right)= |t|^{\gamma(r-p)}\log(1+|t|)^\gamma.$$
		Since $\gamma(r-p)<r$, 
		$$\lim_{|t|\rightarrow +\infty}\frac{(1+|t|)|t|^{\gamma(r-p)}\log(1+|t|)^\gamma}{|t|^{r+1}}=0.$$
		Thus, for every $C>0$, there exists $R>0$ such that
		$$(1+|t|)|t|^{\gamma(r-p)}\log(1+|t|)^\gamma\leq C|t|^{r+1},\quad \mbox{for}\; |t|\geq R,$$
		what gives us
		$$\Gamma\left(x,\frac{F(x,t)}{|t|^p}\right)\leq C\frac{|t|^{r+1}}{1+|t|} \leq C\overline{F}(x,t),\quad \mbox{for}\; |t|\geq R,$$
		that is, the condition $(f_2)$ holds.
		\noindent It is clear that $f$ satisfies the conditions $(f_3)$ and $(f_4)$. Therefore, we are able to apply Theorem \ref{theorem1}.
	
	\begin{remark}
		We point out that, since $\lim_{|t|\rightarrow +\infty}\frac{F(x,t)}{|t|^\theta} = 0$ for all $\theta>r$, then the nonlinearity $f$ does not satisfies Ambrosetti--Rabinowitz condition.
	\end{remark}
	\subsection{Fractional $p(x,\cdot)$-Laplacian operator }\label{ex2}
		Given the function
		$\Phi_{x,y}(t)=\frac{1}{p(x,y)}|t|^{p(x,y)}$ with $p\in C(\Omega\times\Omega)$ symmetric and satisfying the following condition
				$$1<p^{-}\leq p(x,y)\leq p^{+}<N, \quad (x,y)\in \Omega\times\Omega,
		$$
		as a second example we have the well-known fractional $p(x,\cdot)$-Laplacian operator
		$$
		(-\Delta)_{\Phi_{x,y}}^su(x):=(-\Delta)_{p(x,\cdot)}u(x)=2\text{p.v.}\int_{\mathbb{R}^N} \frac{|u(x)-u(y)|^{p(x,y)-2}(u(x)- u(y))}{|x-y|^{N+sp(x,y)}}\;\mathrm{d}y.
		$$

		
		It is not hard to see that $\Phi_{x,y}$ satisfy conditions $(\varphi_1)-(\varphi_3)$ with $\ell=p^{-}$ and $m=p^{+}$. This type of operator has applications in several fields of physics and mathematics, for example, filtration of fluids in porous media, restricted heating, elastoplasticity, image processing, optimal control and financial mathematics, see \cite{Repovs et.al}. For a more comprehensive study of nonlocal problems of this nature, we refer the reader to \cite{Bahrouni and Radulescu,Rossi et. al,Azroul0}.
		
		In this case we consider  $F(x,t):=|t|^{r}\log(1+|t|)$ such that $p^{+}<r <(p^{-})_s^\ast -1$ which gives
		$$\overline{F}(x,t)=\frac{|t|^{r+1}}{1+|t|} +(r-p^+)|t|^{r}\log(1+|t|).$$
		Hence, considering the function $\Gamma(x,t)=|t|^\gamma$ with $\frac{N}{p^{-}}<\gamma<\frac{r}{r-p^{-}}$, 
		it is clear that $f$ satisfies the assumptions $(f_1)-(f_4)$. Therefore, Theorem \ref{theorem1} can be applied to obtain existence of solution of \eqref{eqL}.

	\subsection{Logarithmic perturbation of the $p(x,\cdot)$-Laplacian operator}

	Given the function $\Phi_{x,y}(t)=|t|^{p(x,y)}\log(1+|t|)$ with $p\in C(\Omega\times\Omega)$ symmetric and satisfying the following condition
			$$1<p^{-}\leq p(x,y)\leq p^{+}<N-1, \quad (x,y)\in \Omega\times\Omega,$$ 
	as a third example, we can consider the operator
	$$
	(-\Delta)_{\Phi_{x,y}}^su(x) := 
	2\text{p.v.} \int_{\mathbb{R}^N} \left( p(x,y)|D_su|^{p(x,y)-2}\log(1+|D_su|)+\frac{|D_su|^{p(x,y)-1}}{1+|D_su|}\right) D_su\;\frac{\mathrm{d}y}{|x-y|^{N+s}},
	$$
	and  problem \eqref{eqL} with $F(x,t):=|t|^{r}\log(1+|t|)$ with $p^{+}+1<r<(p^{-})_s^\ast -1$.

		By direct computations, we have
		\begin{equation}\label{eq-5.1}
			t\varphi_{x,y}(t)=\Phi_{x,y}'(t)=p(x,y)t^{p(x,y)-1}\log(1+t)+\frac{t^{p(x,y)}}{1+t}, \quad (x,y)\in\Omega\times\Omega,\; t>0,
		\end{equation}

		When $p(x,y) \equiv p$, $(-\Delta)_{\Phi_{x,y}}^s$ is a fractional version of the logarithmic perturbation of the classical fractional $p$-Laplacian problem. It is clear that $\Phi_{x,y}$ satisfies the assumptions $(\varphi_1)$ and $(\varphi_2)$. It remains to
		show that $(\varphi_3)$ holds. Indeed, note that
		$$p^{-}\leq p(x,y)\leq p(x,y)+\frac{t}{(1+t)\log(1+t)}=\frac{t^2\varphi_{x,y}(t)}{\Phi_{x,y}(t)}, \quad (x,y)\in\Omega\times\Omega,\; t>0.$$
		In addition, 
		$$\lim_{t\rightarrow0^+}\frac{t^2\varphi_{x,y}(t)}{\Phi_{x,y}(t)}=p(x,y)+1\quad \mbox{and}\quad \lim_{t\rightarrow +\infty}\frac{t^2\varphi_{x,y}(t)}{\Phi_{x,y}(t)}=p(x,y), \quad (x,y)\in\Omega\times\Omega.$$
		Thus, since $\frac{t^2\varphi_{x,y}(t)}{\Phi_{x,y}(t)}$ is continuous on $\Omega\times\Omega\times(0,+\infty)$, it follows that 
		$$p^- \leq \frac{t^2\varphi_{x,y}(t)}{\Phi_{x,y}(t)}\leq p^+ + 1, \quad (x,y)\in\Omega\times\Omega,\; t>0.
		$$
		This we conclude that the condition $(\varphi_3)$ is satisfied for $\ell=p^-$ and $m=p^+ +1$.
		As in Example \ref{ex2}, let us consider  $F(x,t):=|t|^{r}\log(1+|t|)$ with $p^{+}+1<r<(p^{-})_s^\ast -1.$ Then, the conditions $(f_1)$ and $(f_2)$ hold. It remains to verify $(f_3)$ and $(f_4)$. Firstly, note that
		$$\frac{f(x,t)}{|t|^{(p^+ +1)-1}}=\frac{f(x,t)}{|t|^{p^+}}=r|t|^{r-p^+ -2}t\log(1+|t|) + \frac{|t|^{r-p^+ -1}t}{1+|t|}.$$
		So, 
		$$\lim_{t\rightarrow+\infty} \frac{f(x,t)}{|t|^{p^+}}=\lim_{t\rightarrow+\infty}r|t|^{r-p^+-1}\log(1+|t|) + \frac{|t|^{r-p^+}}{1+|t|} = +\infty.$$
		This shows that $(f_3)$ holds. Now, for every $x\in\Omega$, we have by \eqref{eq-5.1} that
		$$\begin{aligned}
			\left|\frac{f(x,t)}{|t|\widehat{\varphi}(x,|t|)}\right|= \left|\frac{f(x,t)}{|t|\varphi_{x,x}(|t|)}\right| &\leq \frac{r|t|^{r-1}\log(1+|t|)+\frac{|t|^r}{(1+|t|)}}{p(x,x)|t|^{p(x,x)-1}\log(1+|t|)+\frac{|t|^{p(x,x)}}{1+|t|}}\\
			&\leq\frac{r(1+|t|)|t|^{r-1}\log(1+|t|) + |t|^r}{|t|^{p(x,x)-1}}.
		\end{aligned}$$
		In view from last inequality, it results that
		$$\begin{aligned}
			\lim_{t\rightarrow 0} \left|\frac{f(x,t)}{|t|\widehat{\varphi}(x,|t|)}\right|&\leq \lim_{t\rightarrow 0} \frac{r(1+|t|)|t|^{r-1}\log(1+|t|) + |t|^r}{|t|^{p(x,x)-1}}\\
			& \leq \lim_{t\rightarrow 0} \frac{r(1+|t|)|t|^{r-1}\log(1+|t|) + |t|^r}{|t|^{p+}}
			=0,
		\end{aligned}$$
		uniformly for $x \in \Omega$, and thus $(f_4)$ holds. Therefore, we can apply Theorem \ref{theorem1}.
	
	\vspace{0,5cm}
	
	\begin{acknowledgement}
	 	The first and third authors were partially supported by CNPq with grants 304699/2021-7 and 316643/2021-1, respectively. The second author was partially supported by CAPES.
	\end{acknowledgement}

	
	\bigskip
	\medskip


		\bigskip
		
	\end{document}